\documentclass[12pt,reqno]{amsart}
\usepackage[left=1in,right=1in,top=1in,bottom=1in]{geometry} 
\usepackage{amssymb}
\usepackage{amsmath}
\usepackage{mathrsfs}
\usepackage{amsfonts}
\usepackage{hyperref}
\usepackage{bbm}
\usepackage{tikz}						
\usepackage{xcolor}
\usepackage{multirow}
\usepackage{tikz}
\usepackage{ytableau}
\usepackage{mathrsfs}
\usepackage{thmtools}
\usepackage{thm-restate}
\usepackage{varwidth}
\usepackage{comment}
\usepackage{mathtools}
\usepackage[mathscr]{euscript}
\usepackage{bbm}
\usepackage{multicol}
\usepackage{enumerate}
\usepackage{graphicx} 
\usepackage{csquotes}

\theoremstyle{definition}
\newtheorem{thm}{Theorem}[section]
\newtheorem{lem}[thm]{Lemma}
\newtheorem*{lem*}{Lemma}
\newtheorem*{thm*}{Theorem}
\newtheorem{prop}[thm]{Proposition}

\newtheorem{cor}[thm]{Corollary}

\newtheorem{defn}[thm]{Definition}
\newtheorem*{remark*}{Remark}
\newtheorem{remark}{Remark}
\newtheorem{example}{Example}

\newtheorem{cor/defn}[thm]{Corollary/Definition}

\DeclareMathOperator{\Par}{\mathbb{Y}}

\DeclareMathOperator{\SSYT}{\mathrm{SSYT}}

\DeclareMathOperator{\Dim}{\mathrm{dim}}
\DeclareMathOperator{\Hom}{\mathrm{Hom}}

\DeclareMathOperator{\Char}{\mathrm{char}}
\DeclareMathOperator{\GL}{\mathrm{GL}}
\DeclareMathOperator{\bor}{\mathrm{B}}

\DeclareMathOperator{\tor}{\mathrm{H}}

\DeclareMathOperator{\sort}{sort}
\DeclareMathOperator{\rev}{rev}

\DeclareMathOperator{\mathleft}{left}
\DeclareMathOperator{\mathright}{right}
\DeclareMathOperator{\arm}{arm}
\DeclareMathOperator{\leg}{leg}

\DeclareMathOperator{\mathleg}{lg}

\DeclareMathOperator{\cK}{\mathcal{K}}
\DeclareMathOperator{\cS}{\mathcal{S}}
\DeclareMathOperator{\cP}{\mathcal{P}}
\DeclareMathOperator{\cB}{\mathcal{B}}
\DeclareMathOperator{\cL}{\mathcal{L}}

\title{Multi-Symmetric Schur Functions}
\author{Milo Bechtloff Weising}
\address{Department of Mathematics (0123),
460 McBryde Hall, Virginia Tech,
225 Stanger Street,
Blacksburg, VA 24061-1026}
\email{milojbw@vt.edu}

\date{\today}

\usepackage[
    backend=bibtex,
    maxnames=5,
    style=alphabetic
  ]{biblatex}
\usepackage{graphicx} 

\begin{document}

\begin{abstract}
    We study a multi-symmetric generalization of the classical Schur functions called the multi-symmetric Schur functions. These functions form an integral basis for the ring of multi-symmetric functions indexed by tuples of partitions and are defined as certain stable-limits of key polynomials. We prove combinatorial results about the monomial expansions of the multi-symmetric Schur functions including a diagrammatic combinatorial formula and a triangularity result which completely characterizes their monomial multi-symmetric supports. The triangularity result involves a non-trivial generalization of the dominance order on partitions to tuples of partitions. We prove, using the Demazure character formula, that the multi-symmetric Schur functions expand positively into the basis of tensor products of ordinary Schur functions and describe the expansion coefficients as multiplicities of certain irreducible representations for Levi-subgroups inside particular Demazure modules. Lastly, we find a family of multi-symmetric plethystic operators related to the classical Bernstein operators which act on the multi-symmetric Schur basis by a simple recurrence relation. 
\end{abstract}

\maketitle

\tableofcontents

\section{Introduction}

The \textbf{\textit{Schur functions}} $s_{\lambda}(X)$ form a fundamental basis for the ring of symmetric functions $\Lambda(X)$ indexed by partitions $\lambda.$ Schur functions appear widely across mathematics in representation theory, geometry, and combinatorics. In representation theory, the Schur polynomials $s_{\lambda}(x_1,\ldots,x_n)$ correspond to characters of highest weight polynomial representations of $\mathrm{GL}_n$. In geometry, the Schur polynomials $s_{\lambda}(x_1,\ldots,x_n)$ naturally align with the cohomology classes of Grassmanian Schubert varieties and the Schur functions $s_{\lambda}(X)$ correspond to the cohomology classes of Schubert varieties in the infinite Grassmanian variety. In combinatorics, the Schur functions are a central object for the study of partitions and \textbf{\textit{Young tableaux}}.

The \textbf{\textit{key polynomials}} $\cK_{\alpha}(x)$ are a special family of integer polynomials indexed by compositions forming a basis for the ring $\mathbb{Z}[x_1,x_2,\ldots]$ \cite{RS95}. The key polynomials are the characters of \textbf{\textit{Demazure modules}} in type $\mathrm{GL}_n$ \cite{Dem_1974} \cite{Andersen1985}; $\mathrm{B}_n$-submodules of highest weight $\mathrm{GL}_n$ representations generated by extremal weight vectors. Key polynomials have the benefit of being very easy to work with as they are generated by a simple recursion involving \textbf{\textit{isobaric divided difference operators}}. The key polynomials are related to Schur functions directly by $\cK_{(\lambda_{n},\ldots, \lambda_1)}(x_1,\ldots,x_n) = s_{\lambda}(x_1,\ldots,x_n)$.
Correspondingly, the key polynomials $(\cK_{0^{n-\ell(\lambda)}*\rev(\lambda)}(x_1,\ldots,x_n) )_{n \geq \ell(\lambda)}$ form a \textbf{\textit{stable}} family of polynomials limiting to $s_{\lambda}(X)$ in the sense that $\cK_{0^{n+1-\ell(\lambda)}*\rev(\lambda)}(x_1,\ldots,x_n,0) = \cK_{0^{n-\ell(\lambda)}*\rev(\lambda)}(x_1,\ldots,x_n).$ Here $\rev(\lambda)$ is the composition formed by sorting $\lambda$ into weakly increasing order and $\alpha* \beta$ denotes the concatenation of the compositions $\alpha$ and $\beta$. The key polynomials have their own diagrammatic models and associated combinatorics which relates directly to the tableau combinatorics of Schur functions via these stable-limits \cite{RS95} \cite{Kirillov_2016} \cite{mason_2009} \cite{AQ19}.

However, there are other kinds of stable sequences of key polynomials. In the paper \cite{BWAlmostSymSchur}, the author considers sequences of the form $(\cK_{\mu*0^{n-\ell(\lambda)}*\rev(\lambda)}(x_1,\ldots,x_{n+\ell(\mu)}) )_{n \geq \ell(\lambda)}$ where $\mu$ is an arbitrary composition and $\lambda$ is a partition. These sequences are also stable with limits $s_{(\mu|\lambda)}(X)$ called the \textbf{\textit{almost-symmetric Schur functions}}. For fixed $\ell(\mu)$ the functions $s_{(\mu|\lambda)}(X)$ form a basis of the space of \textbf{\textit{almost-symmetric functions}} which are symmetric after $\ell(\mu).$ In particular, $s_{(\mu|\emptyset)}(X) = \cK_{\mu}(x_1,\ldots,x_{\ell(\mu)})$ and $s_{(\emptyset|\lambda)}(X) = s_{\lambda}(X).$ The almost-symmetric Schur functions are a certain $q,t$ parameter specialization of the stable-limit non-symmetric Macdonald functions considered in \cite{BW}. This generalizes the fact that the non-symmetric Macdonald polynomials $E_{\mu}(x;q,t)$ specialize to the key polynomials $\cK_{\mu}(x)$ after an appropriate $q,t$ substitution \cite{Ion_2003}. The functions $s_{(\mu|\lambda)}(X)$ have a diagrammatic formula and satisfy many nice algebraic and combinatorial properties. 

As we will see in this paper (Theorem \ref{general key stability thm}), the almost-symmetric Schur functions are just one subfamily in a large class of convergent power series formed as stable limits of key polynomials. For any compositions $\beta,\alpha_1,\ldots,\alpha_r$ we may construct a power series
\begin{align*}
    &F_{(\alpha_1|\cdots |\alpha_r)}^{\beta}(y_1,\ldots,y_{\ell(\beta)}|X_1 |z_{1,1},\ldots,z_{1,\ell(\alpha_1)}|\ldots|X_r|z_{r,1},\ldots,z_{r,\ell(\alpha_r)}):= \\
    & \lim_{n_1,\ldots , n_r \rightarrow \infty}\cK_{\beta*0^{n_1}*\alpha_1*\cdots * 0^{n_r}*\alpha_r}(y_1,\ldots,y_{\ell(\beta)}, x_{1,1},\ldots,x_{1,n_{1}},z_{1,1},\ldots,z_{1,\ell(\alpha_1)},\ldots, x_{r,1},\ldots,x_{r,n_r}z_{r,1},\ldots,z_{r,\ell(\alpha_r)}). \\
\end{align*}
These power series live in the partially \textbf{\textit{multi-symmetric}} rings 
$$\mathbb{Z}[y_1,\ldots,y_{\ell(\beta)},z_{1,1},\ldots,z_{1,\ell(\alpha_1)},\ldots,z_{r,1},\ldots,z_{r,\ell(\alpha_r)}]\otimes \Lambda_{\mathbb{Z}}(X_1|\cdots|X_r).$$ We will focus on the case when $\beta = \emptyset$ and each $\alpha_i = \rev(\lambda^{(i)})$ for partitions $\lambda^{(1)},\ldots, \lambda^{(r)}$ and leave the more general case for future work. In this case, we obtain the main object of this paper called the \textbf{\textit{multi-symmetric Schur functions}} $\cS_{(\lambda^{(1)}|\cdots| \lambda^{(r)})}(X_1|\cdots |X_r)$ given by the limit
$$\lim_{n_1,\ldots,n_r \rightarrow \infty}\cK_{(0^{n_1-\ell(\lambda^{(1)})}*\rev(\lambda^{(1)}))*\cdots * (0^{n_r-\ell(\lambda^{(r)})}*\rev(\lambda^{(r)}))}(x_{1,1},\ldots,x_{1,n_1},\ldots,x_{r,1},\ldots,x_{r,n_r}).$$ The partial symmetries of key polynomials have been studied previously. Ross--Yong give a combinatorial formula (Theorem 1.1 \cite{RY15}) for the Schur polynomial expansions of key polynomials $\cK_{\beta}$ which have been split into separate variable sets corresponding to weakly increasing subsequences of $\beta.$ Later works have related these Schur polynomial expansions to the geometry  of Schubert varieties \cite{HY22} \cite{GHY21}.

Here we will outline the main results of this paper. First, we prove a general key polynomial stable-limit theorem \ref{general key stability thm} and apply it to the multi-symmetric situation to obtain Definition \ref{multi-sym Schur function defn}. We readily prove that the family $\{ \cS_{(\lambda^{(1)}|\cdots| \lambda^{(r)})}(X_1|\cdots |X_r) | (\lambda^{(1)}|\cdots | \lambda^{(r)}) \in \Par^r\}$ form a $\mathbb{Q}$-basis for the ring of \textbf{\textit{r multi-symmetric functions}} $\Lambda(X_1|\cdots | X_r).$ Next, we develop the necessary diagrammatics required to describe the combinatorics underlying the functions $\cS_{(\lambda^{(1)}|\cdots| \lambda^{(r)})}$ leading to Theorem \ref{combinatorial formula for multi-sym Schur}. We describe in detail the monomial expansions of the multi-symmetric Schur functions defining multi-symmetric analogues of the \textbf{\textit{Kostka numbers}} $K_{(\mu^{(1)}|\cdots | \mu^{(r)})}^{(\lambda^{(1)}|\cdots| \lambda^{(r)})}$ along with a combinatorial formula (Theorem \ref{Kostka coeff thm}) generalizing the usual shape-content formula for Kostka numbers. We prove two triangularity results for the monomial expansions of the $\cS_{(\lambda^{(1)}|\cdots| \lambda^{(r)})}$. The first is a weaker lexicographic triangularity (Theorem \ref{crude triangularity thm}), which does not fully describe the vanishing of the multi-symmetric Kostka numbers, but has the benefit of being easy to work with. In particular, we will readily conclude that the $\cS_{(\lambda^{(1)}|\cdots| \lambda^{(r)})}$ are actually a $\mathbb{Z}$-basis of $\Lambda_{\mathbb{Z}}(X_1|\cdots |X_r).$ Afterwards, we will define a refined partial order $\triangleleft$ on multi-partitions $\Par^r$ generalizing the classical \textbf{\textit{dominance order}} and show that the multi-symmetric Schur functions are also triangular with respect to $\triangleleft$ (Theorem \ref{complete support of multi-sym Schur}). In fact, we will show that $(\mu^{(1)}|\cdots | \mu^{(r)}) \trianglelefteq (\lambda^{(1)}|\cdots | \lambda^{(r)})$ if and only if $K_{(\mu^{(1)}|\cdots | \mu^{(r)})}^{(\lambda^{(1)}|\cdots| \lambda^{(r)})} \neq 0.$ Next, we use the \textbf{\textit{Demazure character formula}} to show that $\cS_{(\lambda^{(1)}|\cdots| \lambda^{(r)})}$ is Schur positive, i.e., the expansion of each multi-symmetric Schur function into the basis $s_{(\mu^{(1)}|\cdots |\mu^{(r)})}:= s_{\mu^{(1)}}(X_1)\cdots s_{\mu^{(r)}}(X_r)$ has non-negative coefficients (Theorem \ref{Schur positivity theorem}). Moreover, these expansion coefficients are the multiplicities of irreducible representations of type $\GL$ \textbf{\textit{Levi subgroups}} in certain Demazure modules. Lastly, we construct a family of \textbf{\textit{plethystic operators}} $D_1,\ldots,D_{r-1}$ on $\Lambda(X_1|\cdots | X_r)$ using limits of products of isobaric divided difference operators which we show in Theorem \ref{plethysm thm} are given by  
$$D_i(f(X_1|\cdots|X_r))= \langle z^0 \rangle f(X_1|\cdots | X_i+z|X_{i+1}-z^{-1}|\cdots |X_r) \Omega(zX_{i+1}).$$ Finally, we show that
$$D_i\left(\cS_{(\lambda^{(1)}|\cdots |\lambda^{(r)})} \right) = \cS_{(\lambda^{(1)}|\cdots|\lambda^{(i)}\setminus \{\lambda^{(i)}_{1}\}|\lambda^{(i+1)} \cup \{\lambda^{(i)}_{1}\}|\cdots |\lambda^{(r)})}.$$

\subsection{Acknowledgments}
The author would like to thank Daniel Orr for insightful conversations regarding key polynomials and the combinatorics of Macdonald polynomials. The author would also like to thank Alexander Yong for helpful comments regarding a previous version of this paper.

\section{Basic Notions}

\subsection{Key Polynomials}

For the remainder of this paper, all \textit{free variables} are assumed to commute with one another and all commutative rings are assumed to have characteristic $0.$

\begin{defn}
Given a commutative ring $R$ and free variables $x,y$ define the \textbf{\textit{isobaric divided difference operator}} $\xi_{x,y}: R[x,y] \rightarrow R[x,y]$ by 
$$\xi_{x,y}(f(x,y)):= \frac{xf(x,y)-yf(y,x)}{x-y}.$$ For $1\leq i \leq n-1$ we will write $\xi_{i}:= \xi_{x_i,x_{i+1}}$ for the corresponding operator on $\mathbb{Q}[x_1,\ldots,x_n].$
\end{defn}

\begin{defn}\cite{Dem_1974}
    Let $n \geq 1$. Define the \textbf{\textit{key polynomials}} to be the unique collection of polynomials $\{\mathcal{K}_{\alpha}(x_1,\ldots, x_n) \}_{\alpha \in \mathbb{Z}_{\geq 0}^{n}} \subset \mathbb{Q}[x_1,\ldots,x_n]$ determined by the following properties:
    \begin{itemize}
        \item If $\alpha_1 \geq \ldots \geq \alpha_n$ then 
        $$\mathcal{K}_{\alpha}(x_1,\ldots, x_n)= x^{\alpha}.$$
        \item Whenever $\alpha_i > \alpha_{i+1}$
        $$\mathcal{K}_{s_i(\alpha)}(x_1,\ldots, x_n) = \xi_{i}(\mathcal{K}_{\alpha}(x_1,\ldots, x_n)).$$
    \end{itemize}
\end{defn}

We refer the reader to Kirillov \cite{Kirillov_2016}, Mason \cite{mason_2009}, Reiner--Shimozono \cite{RS95}, and Assaf--Quijada \cite{AQ19} for an overview of key polynomials and their associated combinatorics. Below is a fundamental result about key polynomials:

\begin{prop}\cite{RS95}\label{key polys are basis}
    The key polynomials $\{\mathcal{K}_{\alpha}(x_1,\ldots, x_n) \}_{\alpha \in \mathbb{Z}_{\geq 0}^{n}}$ form a $\mathbb{Z}$-basis for $\mathbb{Z}[x_1,\ldots,x_n].$
\end{prop}

 In this paper, we will be using an abridged version of the Haglund--Haiman--Loehr combinatorial model for non-symmetric Macdonald polynomials \cite{haglund2007combinatorial}. 

\begin{defn}\cite{haglund2007combinatorial} \label{HHL defn}
For a composition $\mu = (\mu_1,\dots,\mu_n)$, define the column diagram of $\mu$ as 
$$dg'(\mu):= \{(i,j)\in \mathbb{N}^2 : 1\leq i\leq n, 1\leq j \leq \mu_i \}.$$ This is represented by a collection of boxes in positions given by $dg'(\mu)$. The augmented diagram of $\mu$ is given by 
$$\widehat{dg}(\mu):= dg'(\mu)\cup\{(i,0): 1\leq i\leq n\}.$$
Visually, to get $\widehat{dg}(\mu)$ we are adding a bottom row of boxes on length $n$ below the diagram $dg'(\mu)$. Given $u = (i,j) \in dg'(\mu)$ define the following:
\begin{itemize}
    \item $\leg(u) := \{(i,j') \in dg'(\mu): j' > j\}$
    \item $\arm^{\mathleft}(u) := \{(i',j) \in dg'(\mu): i'<i, \mu_{i'} \leq \mu_i\} $
    \item $\arm^{\mathright}(u):= \{(i',j-1) \in \widehat{dg}(\mu): i'>i, \mu_{i'}<\mu_{i}\}$
    \item $\arm(u) := \arm^{\mathleft}(u) \cup \arm^{\mathright}(u)$
    \item $\mathleg(u):= |\leg(u)| = \mu_i -j$
    \item $a(u) := |\arm(u)|.$
\end{itemize}
A filling of $\mu$ is a function $\sigma: dg'(\mu) \rightarrow \{1,...,n\}$ and given a filling there is an associated augmented filling $\widehat{\sigma}: \widehat{dg}(\mu) \rightarrow \{1,...,n\}$ extending $\sigma$ with the additional bottom row boxes filled according to $\widehat{\sigma}((j,0)) = j$ for $j = 1,\dots,n$. Distinct lattice squares $u,v \in \mathbb{N}^2$ are said to attack each other if one of the following is true:
\begin{itemize}
\item $u$ and $v$ are in the same row 
\item $u$ and $v$ are in consecutive rows and the box in the lower row is to the right of the box in the upper row.
\end{itemize}
A filling $\sigma: dg'(\mu) \rightarrow \{1,\dots,n\}$ is \textbf{\textit{non-attacking}} if $\widehat{\sigma}(u) \neq \widehat{\sigma}(v)$ for every pair of attacking boxes $u,v \in \widehat{dg}(\mu).$
For a box $u= (i,j)$ let $d(u) = (i,j-1)$ denote the box just below $u$.
\end{defn}

\begin{defn}\cite{haglund2007combinatorial}
    Let $\sigma: \mu \rightarrow \{1,\ldots, n\}$ be a non-attacking labeling. A \textbf{\textit{co-inversion triple}} is a triple of boxes $(u,v,w)$ in the diagram $\widehat{dg}(\mu)$ of one of the following two types
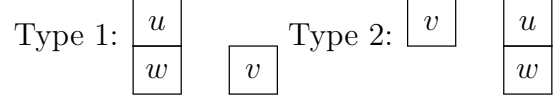
\begin{center}
    Type 1: \begin{ytableau}
        u & \none & \none \\
        w &   \none   & v \\
        \end{ytableau} ~~~~ Type 2: \begin{ytableau}
        v & \none & u \\
        \none &   \none   & w \\
        \end{ytableau}
\end{center}

that satisfy the following criteria:
\begin{itemize}
    \item in Type 1 the column containing $u$ and $w$ is strictly taller than the column containing $v$
    \item in Type 2 the column containing $u$ and $w$ is weakly taller than the column containing $v$
    \item in either Type 1 or Type 2, $\widehat{\sigma}(u)< \widehat{\sigma}(v)< \widehat{\sigma}(w)$ or $\widehat{\sigma}(v)< \widehat{\sigma}(w)< \widehat{\sigma}(u)$ or $\widehat{\sigma}(w)< \widehat{\sigma}(u)< \widehat{\sigma}(v).$
\end{itemize}

Informally, in Type 1 we require the entries to strictly increase clockwise and in Type 2 we require the entries to strictly increase counterclockwise.
\end{defn}

Note that as the diagrams above indicate, in a co-inversion triple we require that $w = d(u)$ and $v \in \arm(u)$. 

\begin{defn}
    Given $\alpha \in \mathbb{Z}_{\geq 0}^{n}$, denote by $\mathcal{L}'(\alpha)$ the set of non-attacking labellings $\sigma: \alpha \rightarrow \{1,\ldots, n\}$ such that $\widehat{\sigma}$ is weakly increasing upwards and has no co-inversion triples. 
\end{defn}

The following is a well-known combinatorial formula for the key polynomials (see \cite{BWAlmostSymSchur} for an explanation):

\begin{prop}\label{key poly combinatorial formula}
    For $\alpha \in \mathbb{Z}_{\geq 0}^{n}$, 
    $$\mathcal{K}_{\alpha} = \sum_{\sigma \in \mathcal{L}'(\alpha) } x^{\sigma}.$$
\end{prop}

\begin{example}
    Consider $\alpha = (1,0,2)$ and the corresponding key polynomial 
    $$\cK_{(1,0,2)}(x_1,x_2,x_3) = x_1^2x_2+x_1^2x_3+x_1x_2^2 + x_1x_2x_3+x_1x_3^2.$$ Each of these monomials corresponds to a diagram filling in the set $\cL'(1,0,2)$ as follows:
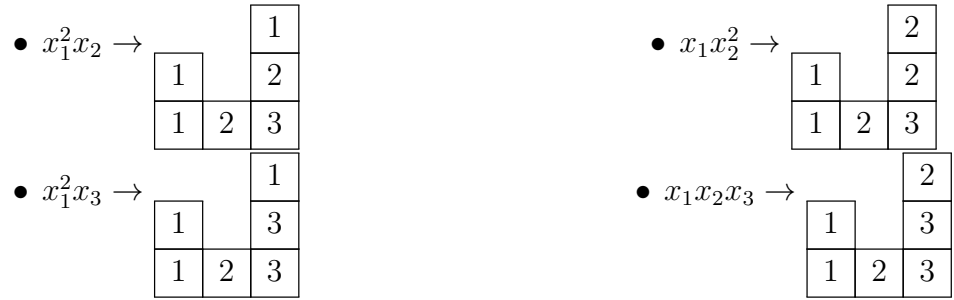
\begin{multicols}{2}
    \begin{itemize}
        \item \begin{center}
    $x_1^2x_2 \rightarrow$ ~~~~ \begin{ytableau}
        \none & \none & 1 \\
        1 &   \none   & 2 \\
        1 & 2 & 3 \\
        \end{ytableau}
        \end{center}
       \item \begin{center}
    $x_1^2x_3 \rightarrow$ ~~~~ \begin{ytableau}
        \none & \none & 1 \\
        1 &   \none   & 3 \\
        1 & 2 & 3 \\
        \end{ytableau}
        \end{center}
       \item \begin{center}
    $x_1x_2^2 \rightarrow$ ~~~~ \begin{ytableau}
        \none & \none & 2 \\
        1 &   \none   & 2 \\
        1 & 2 & 3 \\
        \end{ytableau}
        \end{center}
        \item \begin{center}
    $x_1x_2x_3 \rightarrow$ ~~~~ \begin{ytableau}
        \none & \none & 2 \\
        1 &   \none   & 3 \\
        1 & 2 & 3 \\
        \end{ytableau}
        \end{center}
        \item \begin{center}
    $x_1x_3^2 \rightarrow$ ~~~~ \begin{ytableau}
        \none & \none & 3 \\
        1 &   \none   & 3 \\
        1 & 2 & 3 \\
        \end{ytableau}
        \end{center}
    \end{itemize}
\end{multicols}

It is informative to check that each of these diagram fillings is non-attacking, weakly decreasing upwards, and has no co-inversion triples. For a non-example, the filling
\begin{center}
       \begin{ytableau}
        \none & \none & 1\\
        2* &   \none   & 1* \\
        1 & 2 & 3* \\
        \end{ytableau}
        \end{center}
which is not in the set $\cL'(1,0,2)$ since the boxes with $*$'s form a co-inversion triple of Type 2. More egregiously, this labeling is attacking since the $2$ is row $1$ attacks the $2$ in row $0.$

\end{example}

\subsection{Demazure Characters}

Here we give a brief overview of Demazure characters in type $\GL.$

\begin{defn}
    Let $n \geq 1.$ Define $\bor_n$ to be the Borel subgroup of upper-triangular matrices in $\GL_n$ and let $\tor_n$ denote the group of diagonal matrices in $\GL_n.$ Let $\mathfrak{b}_n$ denote the Lie algebra of $\bor_n$ i.e. the set of upper triangular $n\times n$ matrices over $\mathbb{C}$ with the usual commutator. Let $\mathcal{U}(\mathfrak{b}_n)$ denote the universal enveloping algebra of $\mathfrak{b}_n$. We will write $\mathfrak{S}_n$ for the group of permutations of $\{1,\ldots,n\}$. For a \textbf{\textit{dominant}} (i.e. weakly decreasing) integral weight $\lambda \in \mathbb{Z}_{\geq 0}^n$, let $\mathcal{V}^{\lambda}$ denote the corresponding \textbf{\textit{highest weight representation}} of $\GL_n.$
\end{defn}

\begin{defn}
    Given a finite dimensional polynomial representation $V$ of $\tor_n$, we will denote by $\Char(V) \in \mathbb{Z}[x_1,\ldots,x_n]$ the \textbf{\textit{ formal character}} of $V$ as 
    $$\Char(V) = \sum_{\alpha \in \mathbb{Z}_{\geq 0}^{n}} \Dim \Hom_{\tor_n}(\alpha,V) x^{\alpha}$$ where we denote by $\alpha$ the corresponding character of $\tor_n.$
\end{defn}

\begin{defn}\cite{Dem_1974} \label{Demazure module defn}
    Given a dominant integral weight $\lambda \in \mathbb{Z}_{\geq 0}^{n}$ and $\sigma \in \mathfrak{S}_n$, define the \textbf{\textit{Demazure module}} $\mathcal{V}^{\lambda}_{\sigma(\lambda)}$ to be the $\bor_n$-module $$\mathcal{V}^{\lambda}_{\sigma(\lambda)}:= \mathcal{U}(\mathfrak{b}_n)v$$ where
    $v \in \mathcal{V}^{\lambda}$ is any weight vector with weight $\sigma(\lambda).$ 
\end{defn}

The following formula conjectured by Demazure \cite{Dem_1974} and proved by Andersen \cite{Andersen1985} relates Demazure modules directly to key polynomials.

\begin{thm}(Demazure Character Formula)\cite{Andersen1985} \label{Demazure Character Formula}
    Given a dominant integral weight $\lambda \in \mathbb{Z}_{\geq 0}^{n}$ and $\sigma \in \mathfrak{S}_n$,
    $$\Char(\mathcal{V}^{\lambda}_{\sigma(\lambda)}) = \mathcal{K}_{\sigma(\lambda)}.$$
\end{thm}

\subsection{Symmetric Functions}

\begin{defn}
    Let $\Par$ denote the set of all partitions. We will write $\ell(\lambda)$ for the length of the partition $\lambda$ and $|\lambda|$ for the size of the partition. Given a commutative ring $R$ and an infinite set of free variables $X = \{x_1,x_2,x_3,\ldots\}$ we define $\Lambda_R(X)$ to be the \textit{\textbf{ring of symmetric functions}} over $R$ in the variables $X.$ When $R= \mathbb{Q}$, we simply write $\Lambda(X):= \Lambda_{\mathbb{Q}}(X).$ We will denote by $m_{\lambda}(X)$ and $s_{\lambda}(X)$ respectively the monomial and Schur symmetric functions respectively. Given a collection of infinite sets of free variables we will write $$\Lambda_R(X_1|\cdots|X_r):=\Lambda_R(X_1)\otimes_R \cdots \otimes_R \Lambda_R(X_r)$$ for the \textit{\textbf{ring of $r$ multi-symmetric functions}} over $R$. If $R=\mathbb{Q}$ we write $\Lambda(X_1|\cdots|X_r):= \Lambda_{\mathbb{Q}}(X_1|\cdots | X_r).$ Similarly, we write $$m_{(\lambda^{(1)}|\cdots|\lambda^{(r)})}(X_1|\cdots|X_r):= m_{\lambda^{(1)}}(X_1)\cdots m_{\lambda^{(r)}}(X_r),$$
    and $$s_{(\lambda^{(1)}|\cdots|\lambda^{(r)})}(X_1|\cdots|X_r):= s_{\lambda^{(1)}}(X_1)\cdots s_{\lambda^{(r)}}(X_r)$$ for the $r$ multi-symmetric monomial, and Schur functions respectively.
\end{defn}

Note that the sets $\{m_{(\lambda^{(1)}|\cdots|\lambda^{(r)})}\}_{\underline{\lambda} \in \Par^r}$ and $\{s_{(\lambda^{(1)}|\cdots|\lambda^{(r)})}\}_{\underline{\lambda} \in \Par^r}$ are bases for $\Lambda(X_1|\cdots|X_r).$

\begin{defn}
    We define the \textbf{\textit{plethystic exponential}} as the element of the graded completion of $\Lambda_{\mathbb{Z}}(X)$ given by 
    $$\Omega(X):= \sum_{\lambda} m_{\lambda}(X).$$   
\end{defn}

\section{Stability}

This paper concerns a multi-symmetric generalization of the Schur functions derived from certain special stable-limits of key polynomials. We begin by first examining some general results about the stability of key polynomials and then move to the multi-symmetric situation specifically.

\subsection{Key Polynomial Stability}

The following lemma will be fundamental for the remainder of this paper.

\begin{lem}[Key Polynomial Stability]\label{stability lemma}
    For all compositions $\alpha,\beta$
    $$\cK_{\alpha*0*\beta}(x_1,\ldots, x_{\ell(\alpha)+\ell(\beta)+1})|_{x_{\ell(\alpha)+1}=0} = \cK_{\alpha*\beta}(x_1,\ldots, x_{\ell(\alpha)},x_{\ell(\alpha)+2},\ldots,x_{\ell(\alpha)+\ell(\beta)+1}).$$
\end{lem}
\begin{proof}
    For all $1\leq i \leq \ell(\alpha)+\ell(\beta)+1$ let $\varphi_i(f):= f|_{x_i=0}.$ It is straightforward to check that $\varphi_i\xi_i = s_i \varphi_{i+1}$ for all $1\leq i \leq \ell(\alpha)+\ell(\beta).$ We now directly compute: 
    \begin{align*}
        &\cK_{\alpha*0*\beta}|_{x_{\ell(\alpha)+1}=0} \\
        &= \varphi_{\ell(\alpha)+1} \left( \cK_{\alpha*0*\beta} \right)\\
        &= \varphi_{\ell(\alpha)+1}\xi_{\ell(\alpha)+1}\cdots \xi_{\ell(\alpha)+\ell(\beta)}(\cK_{\alpha*\beta})\\
        &= s_{\ell(\alpha)+1}\cdots s_{\ell(\alpha)+\ell(\beta)} \varphi_{\ell(\alpha)+\ell(\beta)+1}(\cK_{\alpha*\beta})\\
        &= s_{\ell(\alpha)+1}\cdots s_{\ell(\alpha)+\ell(\beta)}(\cK_{\alpha*\beta})\\
        &= \cK_{\alpha*\beta}(x_1,\ldots, x_{\ell(\alpha)},x_{\ell(\alpha)+2},\ldots,x_{\ell(\alpha)+\ell(\beta)+1}).\\
    \end{align*}
\end{proof}

\begin{example}
    We compute directly
    $$\cK_{1001}(x_1,x_2,x_3,x_4) = x_1(x_2+x_3+x_4)$$ and 
    $$\cK_{101}(x_1,x_2,x_3) = x_1(x_2+x_3)$$ so 
    $$\cK_{1001}(x_1,x_2,0,x_4) = \cK_{101}(x_1,x_2,x_4).$$ For another example, consider  
    $$\cK_{110020}(x_1,x_2,x_3,x_4,x_5,x_6) = (x_1^2x_2+x_1x_2^2)(x_3+x_4+x_5) + x_1x_2(x_3^2+x_4^2+x_5^2 + x_3x_4+x_3x_5+x_4x_5)$$ and
    $$\cK_{11020}(x_1,x_2,x_3,x_4,x_5) = (x_1^2x_2+x_1x_2^2)(x_3+x_4) + x_1x_2(x_3^2+x_4^2+ x_3x_4)$$ so 
    $$\cK_{110020}(x_1,x_2,x_3,0,x_5,x_6) = \cK_{11020}(x_1,x_2,x_3,x_5,x_6).$$

    Lastly, consider the simple example 
    $$\cK_{20010010}(x_1,x_2,x_3,x_4,x_5,x_6,x_7,x_8) = x_1^2(x_2x_3+x_2x_4+x_3x_4)+x_1^2(x_2+x_3+x_4)(x_5+x_6+x_7)$$ and 
    $$\cK_{2010010}(x_1,x_2,x_3,x_4,x_5,x_6,x_7) = x_1^2x_2x_3 + x_1^2(x_2+x_3)(x_4+x_5+x_6)$$ so 
    $$\cK_{20010010}(x_1,x_2,0,x_4,x_5,x_6,x_7,x_8) =\cK_{2010010}(x_1,x_2,x_4,x_5,x_6,x_7,x_8).$$
\end{example}

Applying Lemma \ref{stability lemma} yields the following result:

\begin{thm}\label{general key stability thm}
    For any compositions $\beta,\alpha_1,\ldots,\alpha_r$ the following limit exists as a formal power series:
    \begin{align*}
        &F_{(\alpha_1|\cdots |\alpha_r)}^{\beta} :=\\
        &\lim_{n_1,\ldots , n_r \rightarrow \infty}\cK_{\beta*0^{n_1}*\alpha_1*\cdots * 0^{n_r}*\alpha_r}(y_1,\ldots,y_{\ell(\beta)}, x_{1,1},\ldots,x_{1,n_{1}},z_{1,1},\ldots,z_{1,\ell(\alpha_1)},\ldots, x_{r,1},\ldots,x_{r,n_r}z_{r,1},\ldots,z_{r,\ell(\alpha_r)}).
    \end{align*}

\end{thm}

We, in fact, find that 
$F_{(\alpha_1|\cdots |\alpha_r)}^{\beta}(y_1,\ldots,y_{\ell(\beta)}|X_1 |z_{1,1},\ldots,z_{1,\ell(\alpha_1)}|\ldots|X_r|z_{r,1},\ldots,z_{r,\ell(\alpha_r)})$ is an element of the ring 
$$\mathbb{Z}[y_1,\ldots,y_{\ell(\beta)},z_{1,1},\ldots,z_{1,\ell(\alpha_1)},\ldots,z_{r,1},\ldots,z_{r,\ell(\alpha_r)}]\otimes \Lambda_{\mathbb{Z}}(X_1|\cdots|X_r).$$ The case of when $\beta$ is arbitrary, $r=1$, and $\alpha = \rev(\lambda)$ for a partition $\lambda$ was considered in the author's previous work on almost symmetric Schur functions $s_{(\mu|\lambda)}(X)$ \cite{BWAlmostSymSchur}. Moreover, 
$$F^{\mu}_{\rev(\lambda)}(y|X) = s_{(\mu|\lambda)}(y_1,\ldots,y_{\ell(\mu)},x_1,x_2,\ldots).$$ In general, the limits \ref{general key stability thm} are complicated but when $\beta = \emptyset$ and $(\alpha_1  |\cdots |\alpha_r)=(\rev(\lambda^{(1)})|\cdots | \rev(\lambda^{(r)}))$ for partitions $\lambda^{(1)},\ldots,\lambda^{(r)}$, we see that
$F_{(\rev(\lambda^{(1)})|\cdots | \rev(\lambda^{(r)}))}^{\emptyset}$ lives in the ring 
$\Lambda_{\mathbb{Z}}(X_1|\cdots|X_r).$ In this paper, we will focus on this special case and discuss the general situation in a later work.

\begin{example}
Here is a simple partially-symmetric example concerning the composition $\alpha = (2,1):$
    \begin{itemize}
        \item $\cK_{21} = x_1^2x_2$
        \item $\cK_{021} = x_1^2x_2 + x_1x_2^2 + x_1^2x_3 + x_1x_2x_3 + x_2^2x_3$
        \item $\cK_{0021} = x_1^2x_2 + x_1x_2^2 + (x_1^2+x_2^2)(x_3+x_4) + x_1x_2(2x_3+x_4) + (x_1+x_2)(x_3^2+x_3x_4) + x_3^2x_4$
        \item $\cK_{00021} = x_1^2x_2 + x_1x_2^2 +x_2^2x_3 + x_1^2x_3+x_1x_3^2+x_2x_3^2 + 2x_1x_2x_3 + (x_1^2+x_2^2+x_3^2+x_1x_2+x_1x_3+x_2x_3)(x_4+x_5) +(x_1x_2+x_1x_3+x_2x_3)x_4 + (x_1+x_2+x_3)(x_4^2+2x_4x_5) + x_4^2x_5$
    \end{itemize}
    Note that we may rewrite $\cK_{00021}$ in the simplified form
    $$\cK_{00021} = s_{2,1}(x_1,x_2,x_3)+s_2(x_1,x_2,x_3)(x_4+x_5) + s_{1,1}(x_1,x_2,x_3)x_4 + s_1(x_1,x_2,x_3)(x_4^2+x_4x_5) + x_4^2x_5$$
    which shows that 
    $$F^{\emptyset}_{(2,1)}(X|z_1,z_2) = s_{2,1}(X)+s_2(X)(z_1+z_2) + s_{1,1}(X)z_1 + s_1(X)(z_1^2+z_1z_2) + z_1^2z_2.$$
    In fact, this function has more symmetry:
    $$F^{\emptyset}_{(2,1)}(X|z_1,z_2) = s_{2,1}(X+z_1)+s_2(X+z_1)z_2.$$ 
    This property may be read off directly from the compositions $0^n*(2,1)$ since the full substring $0^n*2$ is weakly increasing and not just the substring $0^n$. In general, the $F_{(\alpha_1|\cdots |\alpha_r)}^{\beta}$ have similar additional symmetries whenever $(\alpha_1|\cdots |\alpha_r) \neq \emptyset.$
\end{example}

\begin{example}
\label{2-multisym example of key poly}    Here we give a multi-symmetric example concerning the partitions $\lambda^{(1)} = 2$ and $\lambda^{(2)} = 1:$
    \begin{itemize}
        \item $\cK_{21} = x_1^2x_2$
        \item $ \cK_{0201} = x_1^2x_2+x_1x_2^2 + (x_1^2+x_1x_2+x_2^2)(x_3+x_4)$
        \item $\cK_{002001}= x_1^2x_2 + x_1x_2^2 +x_2^2x_3 + x_1^2x_3+x_1x_3^2+x_2x_3^2 + 2x_1x_2x_3 + (x_1^2+x_2^2+x_3^2+x_1x_2+x_1x_3+x_2x_3)(x_4+x_5+x_6).$
    \end{itemize}
    This means that 
    $$F^{\emptyset}_{(2|1)}(X_1|X_2) = s_{2}(X_1)s_1(X_2) + s_{2,1}(X_1).$$ It is not a coincidence that this example is highly related to the example $F^{\emptyset}_{(2,1)}(X|z_1,z_2)$. In general, there are many algebraic relationships between the functions $F_{(\alpha_1|\cdots |\alpha_r)}^{\beta}$ which will be explored in a future paper.
\end{example}

\subsection{Multi-symmetric Limits}

As we will be focusing on the multi-symmetric limits of key polynomials specifically, it will be helpful to establish some notation that will be suited to that context.

\begin{defn}\label{multi-symmetric Schur poly defn}
    Let $\cP_{(n_1|\cdots|n_r)}:= \bigotimes_{i=1}^{r} \mathbb{Q}[x_{i,1},\ldots,x_{i,n_i}].$ We let $ \mathfrak{S}^{(1)}_{n_1}\times \cdots \times \mathfrak{S}^{(r)}_{n_r}$ act on $\cP_{(n_1|\cdots|n_r)}$ via permutations of the variables $x_{i,1},\ldots,x_{i,n_i}$ where each $\mathfrak{S}^{(i)}_{n_i}$ is a copy of $\mathfrak{S}_{n_i}$. We will write $\Lambda(X_{1}^{(n_1)}|\cdots|X_{r}^{(n_r)}):= \cP_{(n_1|\cdots|n_r)}^{\mathfrak{S}^{(1)}_{n_1}\times \cdots \times \mathfrak{S}^{(r)}_{n_r}}$ where each $X_{i}^{(n_i)}$ denotes a symmetrized variable set $X_{i}^{(n_i)} = x_{i,1}+\ldots+x_{i,n_i}.$ Let $(\lambda^{(1)}|\ldots|\lambda^{(r)})$ be an r-tuple of partitions. For all $(n_1|\cdots|n_r)$ with $n_i \geq \ell(\lambda^{(i)})$ define the \textbf{\textit{multi-symmetric Schur polynomial}} $$\cS_{(\lambda^{(1)}|\ldots|\lambda^{(r)})}^{(n_1|\cdots|n_r)}(X_{1}^{(n_1)}|\cdots|X_{r}^{(n_r)}):= \cK_{(0^{n_1-\ell(\lambda^{(1)})}*\rev(\lambda^{(1)}))*\cdots * (0^{n_r-\ell(\lambda^{(r)})}*\rev(\lambda^{(r)}))}(x_{1,1},\ldots,x_{1,n_1},\ldots,x_{r,1},\ldots,x_{r,n_r}).$$  
\end{defn}

Note that $\cS_{(\lambda^{(1)}|\ldots|\lambda^{(r)})}^{(n_1|\cdots|n_r)}(X_{1}^{(n_1)}|\cdots|X_{r}^{(n_r)})$ is indeed well-defined as a polynomial in the symmetrized variable sets $X_{1}^{(n_1)},\ldots,X_{r}^{(n_r)}$ since the composition $(0^{n_1-\ell(\lambda^{(1)})}*\rev(\lambda^{(1)}))*\cdots * (0^{n_r-\ell(\lambda^{(r)})}*\rev(\lambda^{(r)}))$ is weakly increasing along each of the windows $\{n_1+\ldots + n_{i-1} +1,\ldots,n_1+\ldots + n_{i-1} +n_{i}\}$ for $1 \leq i \leq r.$

\begin{defn}
    For $1\leq i \leq r$ and $n_i \geq 1$ we define the map 
    $\Xi^{(n_i)}_i: \cP_{(n_1|\cdots|n_r)} \rightarrow \cP_{(n_1|\cdots|n_{i-1}|n_i-1|n_{i+1}|\cdots|n_r)}$ by 
    $$\Xi^{(n_i)}_i(f) := f|_{x_{i,n_i}=0}.$$ 
\end{defn}

We may restate Lemma \ref{stability lemma} in this context as follows.

\begin{lem}\label{mutlisym stability lemma}
    For all $(\lambda^{(1)}|\cdots |\lambda^{(r)}) \in \Par^r$ and $(n_1|\cdots|n_r)$ with $n_j \geq \ell(\lambda^{(j)})$ for all $1\leq j \leq r$ and $n_i \geq \ell(\lambda^{(i)})+1$ for some $1 \leq i \leq r$, 
    $$\Xi_{i}^{(n_i)}\left(\cS_{(\lambda^{(1)}|\cdots |\lambda^{(r)})}^{(n_1|\cdots|n_r)} \right) = \cS_{(\lambda^{(1)}|\cdots |\lambda^{(r)})}^{(n_1|\cdots|n_{i-1}|n_i-1|n_{i+1}|\cdots|n_r)}.$$ 
\end{lem}

We may now define the main object of this paper.

\begin{defn}\label{multi-sym Schur function defn}
    For $(\lambda^{(1)}|\cdots |\lambda^{(r)}) \in \Par^r,$ define the \textbf{\textit{multi-symmetric Schur function}} 
    $$\cS_{(\lambda^{(1)}|\cdots |\lambda^{(r)})}(X_1 | \cdots |X_r) := \lim_{n_1,\ldots,n_r \rightarrow \infty} \cS_{(\lambda^{(1)}|\cdots |\lambda^{(r)})}^{(n_1|\cdots|n_r)}(X_{1}^{(n_1)} | \cdots |X_{r}^{(n_r)}).$$
\end{defn}

Note that the multi-symmetric Schur function $\cS_{(\lambda^{(1)}|\cdots |\lambda^{(r)})}(X_1 | \cdots |X_r)$ is homogeneous of total degree $|\lambda^{(1)}|+\ldots + |\lambda^{(r)}|.$ However, the $\cS_{(\lambda^{(1)}|\cdots |\lambda^{(r)})}$ are not homogeneous with respect to the finer grading in each variable set separately as the following examples demonstrate. 

\begin{example}
Below are the first few examples of multi-symmetric Schur functions expanded into the Schur basis for $r = 2$:
    \begin{multicols}{2}
    \begin{itemize}
        \item $\cS_{(\emptyset|\emptyset)} = 1$
        \item $\cS_{(1|\emptyset)} = s_{(1|\emptyset)}$
        \item $\cS_{(\emptyset|1)} = s_{(\emptyset|1)} + s_{(1|\emptyset)} $
        \item $\cS_{(2|\emptyset)} = s_{(2|\emptyset)}$
        \item $\cS_{(1,1|\emptyset)} = s_{(1,1|\emptyset)}$
        \item $\cS_{(1|1)} = s_{(1|1)}+s_{(1,1|\emptyset)}$
        \item $\cS_{(\emptyset| 2)} = s_{(\emptyset|2)}+ s_{(1|1)} + s_{(2|\emptyset)}$
         \item $\cS_{(\emptyset|1,1)} = s_{(\emptyset|1,1)} + s_{(1|1)}+s_{(1,1|\emptyset)}$
        \item $\cS_{(3|\emptyset)} = s_{(3|\emptyset)}$
        \item $\cS_{(2,1|\emptyset)} = s_{(2,1|\emptyset)}$
        \item $\cS_{(1,1,1|\emptyset)} = s_{(1,1,1|\emptyset)}$
        \item $\cS_{(2|1)} = s_{(2|1)}+s_{(2,1|\emptyset)}$
        \item $\cS_{(1,1|1)} = s_{(1,1|1)}+s_{(1,1,1|\emptyset)}$
        \item $\cS_{(1|2)} = s_{(1|2)}+s_{(1,1|1)}+s_{(2|1)}+s_{(2,1|\emptyset)}$
        \item $\cS_{(1|1,1)} = s_{(1|1,1)}+s_{(1,1,1|\emptyset)} + s_{(1,1|1)}$
        \item $\cS_{(\emptyset|3)} = s_{(\emptyset|3)} + s_{(1|2)}+s_{(2|1)}+s_{(3|\emptyset)}$
        \item $\cS_{(\emptyset|2,1)}= s_{(\emptyset|2,1)}+s_{(1|1,1)} +s_{(1|2)}+s_{(1,1|1)}+s_{(2|1)}+s_{(2,1|\emptyset)}$
        \item $\cS_{(\emptyset|1,1,1)}= s_{(\emptyset|1,1,1)}+s_{(1|1,1)}+s_{(1,1|1)}+s_{(1,1,1|\emptyset)}$
        
    \end{itemize}
    \end{multicols}

Notice that the above expansions are all non-negative in the ordinary Schur basis $s_{(\mu^{(1)}|\cdots | \mu^{(r)})}$. Later, we will prove this holds in general (Theorem \ref{Schur positivity theorem}). Generally, the expansion of the $\cS_{(\lambda^{(1)}|\cdots | \lambda^{(r)})}$ into the $s_{(\mu^{(1)}|\cdots | \mu^{(r)})}$ are not multiplicity free. For example, $s_{(2,1|2,1)}$ occurs with coefficient $2$ in $\cS_{(\emptyset|3,2,1)}.$ It is not a coincidence that this is the simplest \textbf{\textit{Littlewood-Richardson coefficient}} $c_{(2,1),(2,1)}^{(3,2,1)} = 2$ which is larger than $1$. See Remark \ref{LR remark} for more details.
\end{example}

\begin{remark}
    The $r=1$ case recovers the usual Schur functions
    $$\cS_{\lambda}(X) = s_{\lambda}(X).$$ Note that in terms of the more general notation in Theorem \ref{general key stability thm},
    $$\cS_{(\lambda^{(1)}|\cdots |\lambda^{(r)})}(X_1|\cdots |X_r) = F^{\emptyset}_{(\rev(\lambda^{(1)})|\cdots |\rev(\lambda^{(r)}))}(X_1|\cdots|X_r)$$ where we have set the $z$-variables to $0.$ Whenever some $\lambda^{(i)} = \emptyset$, we also have the simple relation:
    $$\cS_{(\lambda^{(1)}|\cdots|\emptyset|\cdots|\lambda^{(r)})}(X_1|\cdots|X_i|X_{i+1}|
    \cdots|X_r) =\cS_{(\lambda^{(1)}|\cdots|\lambda^{(i+1)}|\cdots|\lambda^{(r)})}(X_1|\cdots|X_i+X_{i+1}|\cdots|X_r).$$
\end{remark}

Similar to the ordinary Schur functions, the multi-symmetric Schur functions form a basis for the ring of multi-symmetric functions. We first show this result over $\mathbb{Q}$ and deal with the more subtle situation over $\mathbb{Z}$ later in Corollary \ref{Z basis Cor.}. 

\begin{cor}
    The multi-symmetric Schur functions $\{\cS_{(\lambda^{(1)}|\cdots |\lambda^{(r)})}| (\lambda^{(1)}|\cdots |\lambda^{(r)})\in\Par^r\}$ form a $\mathbb{Q}$-basis for $\Lambda(X_1|\cdots|X_r).$
\end{cor}
\begin{proof}
    There are sufficiently many $\cS_{(\lambda^{(1)}|\cdots |\lambda^{(r)})}$ in each degree so it suffices to show that they are linearly independent; but this is a consequence of the stability in Lemma \ref{mutlisym stability lemma} combined with Proposition \ref{key polys are basis}.
\end{proof}

\section{Monomial Expansion}

This section focuses on analyzing the monomial expansions of the multi-symmetric Schur functions. First, we give a diagrammatic rule for their monomial expansions (Theorem \ref{combinatorial formula for multi-sym Schur}) derived from the known combinatorial formula for key polynomials (Proposition \ref{key poly combinatorial formula}). We use this to give a combinatorial formula for the coefficients of the monomial multi-symmetric function expansions of the multi-symmetric Schur functions (Theorem \ref{Kostka coeff thm}). Next, we give two different partial orders on $\Par^r$, $\prec$ and $\triangleleft$, which describe the triangularity of the monomial multi-symmetric expansions of the multi-symmetric Schur functions (Theorems \ref{crude triangularity thm} and \ref{complete support of multi-sym Schur}). As result, we show that the multi-symmetric Schur functions are a $\mathbb{Z}$-basis of $\Lambda_{\mathbb{Z}}(X_1|\cdots |X_r).$

\subsection{Combinatorial Formula}

In order to state and prove the combinatorial formula Theorem \ref{combinatorial formula for multi-sym Schur}, we will need to introduce a few combinatorial constructions.

\begin{defn}
    Given $(\lambda^{(1)}|\cdots |\lambda^{(r)}) \in \Par^r$ and $(n_1|\cdots|n_r)$ with $n_i \geq \ell(\lambda^{(i)})$ for all $1\leq i \leq r,$ define the diagrams 
    $$dg'_{(n_1|\cdots|n_r)}(\lambda^{(1)}|\cdots|\lambda^{(r)}):= dg'((0^{n_1-\ell(\lambda^{(1)})}*\rev(\lambda^{(1)}))*\cdots * (0^{n_r-\ell(\lambda^{(r)})}*\rev(\lambda^{(r)})))$$
    $$\widehat{dg}_{(n_1|\cdots|n_r)}(\lambda^{(1)}|\cdots|\lambda^{(r)}):= \widehat{dg}((0^{n_1-\ell(\lambda^{(1)})}*\rev(\lambda^{(1)}))*\cdots * (0^{n_r-\ell(\lambda^{(r)})}*\rev(\lambda^{(r)}))).$$
    We will require multiple copies of the usual labels $\{1,2,3,\ldots\}$. We consider labels of the form $i_j$ for $i,j \geq 1$ which we order by $i_j \leq k_{\ell}$ if either $j < \ell$ or $j = \ell$ and $i < k.$ We naturally extend the definition of co-inversion triples to any lattice diagram with labels of the form $i_j$ for $i,j \geq 1$ using the ordering just described. Given any such labeling $\sigma$ we define the monomial 
    $$x^{\sigma} = \prod_{j \geq 1} \prod_{ i \geq 1} x_{j,i}^{|\sigma^{-1}(i_j)|}.$$ We denote by $[n_1|\ldots|n_r]$ the set $\{1_1,\ldots,{(n_1)}_1\} \cup \ldots \cup \{1_r,\ldots,{(n_r)}_r \}.$ Given a labeling $$ \sigma:dg'((0^{n_1-\ell(\lambda^{(1)})}*\rev(\lambda^{(1)}))*\cdots * (0^{n_r-\ell(\lambda^{(r)})}*\rev(\lambda^{(r)}))) \rightarrow [n_1|\ldots|n_r]$$ and define 
    $$\widehat{\sigma}:\widehat{dg}((0^{n_1-\ell(\lambda^{(1)})}*\rev(\lambda^{(1)}))*\cdots * (0^{n_r-\ell(\lambda^{(r)})}*\rev(\lambda^{(r)}))) \rightarrow [n_1|\ldots|n_r]$$ by extending $\sigma$ to the basement boxes via 
    $$\widehat{\sigma}(n_1+\ldots+n_{j-1}+i,0):= i_j.$$ We will denote by $\cL^{(n_1|\cdots|n_r)}(\lambda^{(1)}|\cdots|\lambda^{(r)})$ the set of non-attacking labellings $$ \sigma:dg'((0^{n_1-\ell(\lambda^{(1)})}*\rev(\lambda^{(1)}))*\cdots * (0^{n_r-\ell(\lambda^{(r)})}*\rev(\lambda^{(r)}))) \rightarrow [n_1|\ldots|n_r]$$ such that $\widehat{\sigma}$ is weakly decreasing upwards and has no co-inversion triples. 
\end{defn}

Using Proposition \ref{key poly combinatorial formula}, we find that 
\begin{equation}\label{combinatorial formula eq}
    \cS_{(\lambda^{(1)}|\cdots |\lambda^{(r)})}^{(n_1|\cdots|n_r)} = \sum_{\sigma \in \cL^{(n_1|\cdots|n_r)}(\lambda^{(1)}|\cdots|\lambda^{(r)})} x^{\sigma}.
\end{equation}

\begin{remark}
    The restriction of each labeling $\sigma \in \cL^{(n_1|\cdots|n_r)}(\lambda^{(1)}|\cdots|\lambda^{(r)})$ to one of the diagram components $\rev(\lambda^{(i)})$ is weakly decreasing upwards and strictly increasing rightwards along connected segments of rows.
\end{remark}

\begin{defn}
    Given a labeling $\sigma$ of a lattice diagram, we say that $\sigma$ is inverse semi-standard if $\sigma$ weakly decreases upwards and strictly increasing rightwards along rows. We will write 
    $$\mathbb{N}(r):= \{1_1,2_1,\ldots\}\cup \ldots \cup \{1_r,2_r,\ldots \}.$$ We define $\mathbb{N}(r)_{\infty}$ by extending the poset $\mathbb{N}(r)$ by adding in elements   $\omega_1,\ldots,\omega_r$ and formal sums $\omega_i + j$ for $1\leq i \leq r$ and $j \geq 0$ which we order as follows:
    $$a_1 < \omega_1 + b_1 < c_2 < \omega_2 + d_2 < \ldots < e_r < \omega_r + f_r$$ for any $a,b,c,d,e,f \in \{1,2,3\ldots\}.$ Note the subscripts here denote which alphabet those values belong in. Informally, everything within each extended alphabet is ordered as one would expect and each extended alphabet is mutually less than the next one. Given a labeling $\sigma: dg'(\rev(\lambda^{(1)})*\ldots * \rev(\lambda^{(r)})) \rightarrow \mathbb{N}(r)$ define $\sigma^{\star}:\widehat{dg}(\rev(\lambda^{(1)})*\cdots * \rev(\lambda^{(r)}))$ by extending $\sigma$ and defining $\sigma^{\star}$ on the basement boxes via 
    $$\sigma^{\star}(\ell(\lambda^{(1)})+\ldots + \ell(\lambda^{(i-1)}) + j,0):= \omega_i + j -1 $$ for $1\leq i \leq r$ and $1 \leq j \leq \ell(\lambda^{(i)}).$
    Define $\cL(\lambda^{(1)}|\cdots |\lambda^{(r)})$ to be the set of labellings $\sigma: \rev(\lambda^{(1)})*\ldots * \rev(\lambda^{(r)}) \rightarrow \mathbb{N}(r)$ such that $\sigma^{\star}$ 
    is non-attacking, weakly decreasing upwards, and has no co-inversion triples.
\end{defn}

When considering fillings for multiple partitions, we will signify the distinct components using color-coding. Below are the diagrams corresponding to the partitions $(2,1|2,1,1|4,4,4,3,2,2,2):$
\begin{center} $dg'(2,1|2,1,1|4,4,4,3,2,2,2) = $
 \ytableausetup
 {mathmode, boxframe=normal, boxsize=3em}
\begin{ytableau}
        \none & \none & \none & \none & \none & \none & \none & \none & \none & \none & *(green) &*(green) &*(green) \\
        \none & \none & \none & \none & \none & \none & \none & \none & \none & *(green) & *(green) &*(green) &*(green) \\
        \none & \none & *(red) & \none & \none & *(yellow) & *(green) & *(green) & *(green) & *(green) & *(green) & *(green) & *(green) \\
        \none &   *(red)   & *(red) & *(yellow) & *(yellow) & *(yellow) & *(green) & *(green) & *(green) & *(green) & *(green) & *(green) & *(green)\\ 
\end{ytableau}
\end{center}

\begin{center} $\widehat{dg}(2,1|2,1,1|4,4,4,3,2,2,2) = $
\begin{ytableau}
        \none & \none & \none & \none & \none & \none & \none & \none & \none & *(green) &*(green) &*(green) \\
         \none & \none & \none & \none & \none & \none & \none & \none & *(green) & *(green) &*(green) &*(green) \\
         \none & *(red) & \none & \none & *(yellow) & *(green) & *(green) & *(green) & *(green) & *(green) & *(green) & *(green) \\
            *(red)   & *(red) & *(yellow) & *(yellow) & *(yellow) & *(green) & *(green) & *(green) & *(green) & *(green) & *(green) & *(green)\\ 
       *(red) \omega_1 & *(red) \omega_1 + 1 & *(yellow) \omega_2 & *(yellow) \omega_2 + 1 & *(yellow) \omega_2 +2 & *(green) \omega_3 & *(green) \omega_3 +1 & *(green) \omega_3 + 2 & *(green) \omega_3 + 3 & *(green) \omega_3 + 4 & *(green) \omega_3 + 5 & *(green) \omega_3 +6  \\
\end{ytableau}
\end{center}

\begin{example}
    Here is a filling $\sigma \in \cL(2,1|2,1,1|4,4,4,3,2,2,2)$:

\begin{center}
\begin{ytableau}
          \none & \none & \none & \none & \none & \none & \none & \none & \none & *(green) 1_1 &*(green) 4_1 &*(green) 4_3\\
         \none & \none & \none & \none & \none & \none & \none & \none & *(green) 9_1 & *(green)3_2 &*(green) 4_2&*(green)5_3 \\
         \none & *(red) 10_1 & \none & \none & *(yellow) 5_1 & *(green) 6_1 & *(green) 7_1 & *(green) 8_1 & *(green) 9_1 & *(green) 3_2 & *(green)4_2 & *(green) 5_3\\
           *(red) 9_1 & *(red) 10_1 & *(yellow) 3_1 & *(yellow) 7_1 & *(yellow) 7_2 & *(green) 5_2 & *(green) 8_2 & *(green) 1_3 & *(green) 2_3 & *(green) 3_3 & *(green) 4_3 &  *(green) 5_3\\ 
           *(red) \omega_1 & *(red) \omega_1 + 1 & *(yellow) \omega_2 & *(yellow) \omega_2 + 1 & *(yellow) \omega_2 +2 & *(green) \omega_3 & *(green) \omega_3 +1 & *(green) \omega_3 + 2 & *(green) \omega_3 + 3 & *(green) \omega_3 + 4 & *(green) \omega_3 + 5 & *(green) \omega_3 +6  \\
\end{ytableau}
\end{center}
 This filling has weight 
 $$x_{\sigma} = x_{1,1}x_{1,3}x_{1,4}x_{1,5}x_{1,6}x_{1,7}^2x_{1,8}x_{1,9}^3x_{1,10}^2 x_{2,3}^2x_{2,4}^2 x_{2,5}x_{2,7}x_{2,8}x_{3,1}x_{3,2}x_{3,3}x_{3,4}x_{3,5}^3.$$
    
\end{example}

The following combinatorial formula for the multi-symmetric Schur functions generalizes the well-known combinatorial for the Schur functions.

\begin{thm}\label{combinatorial formula for multi-sym Schur}
    $$\cS_{(\lambda^{(1)}|\cdots|\lambda^{(r)})}(X_1|\cdots|X_r) = \sum_{\sigma \in \cL(\lambda^{(1)}|\cdots|\lambda^{(r)})} x^{\sigma}.$$
\end{thm}
\begin{proof}
   The result follows by applying the stability from Lemma \ref{mutlisym stability lemma} to \ref{combinatorial formula eq}. In particular, 
    $$\cL(\lambda^{(1)}|\cdots|\lambda^{(r)}) \equiv \bigcup_{\substack{n_1,\ldots,n_r\\ n_i \geq \ell(\lambda^{(i)})} } \cL^{(n_1|\cdots|n_r)}(\lambda^{(1)}|\cdots|\lambda^{(r)})$$ by the way of the obvious weight preserving bijections.
\end{proof}

\begin{remark}
    In the case that $r=1$, $\cL(\lambda)$ is in natural bijection with $\SSYT(\lambda).$ The fillings in $\cL(\lambda)$ are exactly those labellings of $\rev(\lambda)$ which are strictly increasing rightwards and weakly decreasing upwards. We may reflect the diagram $ \rev(\lambda)$ along with the labeling along any positive slope diagonal above the diagram to obtain a filling of $\lambda$ (in English notation) which is weakly decreasing rightwards along rows and strictly decreasing down columns. Now order the entries in the filling $a_1 < \ldots < a_r$ and apply the order-reversing permutation $a_i \rightarrow a_{r-i+1}.$ The resulting filling will be a semi-standard Young tableau of shape $\lambda.$ This map defines a bijection $\cL(\lambda) \leftrightarrow \SSYT(\lambda).$ However, this map is not weight preserving. At first this appears to be an issue since $\cK_{\rev(\lambda)}(x_1,\ldots,x_n) = s_{\lambda}(x_1,\ldots, x_n)$ for all $n$. Let $\cL(\lambda;n)$ and $  \SSYT(\lambda;n)$ denote the set of allowable fillings with labels restricted to the sets $\{1,\ldots, n\}$ for both $\rev(\lambda)$ and $\lambda$ respectively. To alleviate this seeming discrepancy we may define weight preserving bijections $\cL(\lambda;n)\leftrightarrow \SSYT(\lambda;n)$ by repeating the above reflection process and instead applying the longest permutation in $\omega_0^{(n)} \in \mathfrak{S}_n$ i.e. $\omega_0^{(n)}(i):= n-i+1$ to the values. These bijections are not compatible as $n$ increases but, these verify that $\cK_{\rev(\lambda)}(x_1,\ldots,x_n) = s_{\lambda}(x_1,\ldots, x_n)$ diagrammatically for every $n.$
\end{remark}

\subsection{Kostka Coefficients}

We may define multi-symmetric analogues of the usual Kostka numbers. 

\begin{defn}
    We define the \textbf{\textit{multi-symmetric Kostka coefficients}} $K_{(\mu^{(1)}|\cdots|\mu^{(r)})}^{(\lambda^{(1)}|\cdots|\lambda^{(r)})}$ for $(\mu^{(1)}|\cdots|\mu^{(r)}),(\lambda^{(1)}|\cdots|\lambda^{(r)}) \in \Par^r$ by 
    $$\cS_{(\lambda^{(1)}|\cdots|\lambda^{(r)})} = \sum_{(\mu^{(1)}|\cdots|\mu^{(r)}) \in \Par^r} K_{(\mu^{(1)}|\cdots|\mu^{(r)})}^{(\lambda^{(1)}|\cdots|\lambda^{(r)})} m_{(\mu^{(1)}|\cdots|\mu^{(r)})}.$$
\end{defn}

\begin{defn}
    Given $(\mu^{(1)}|\cdots|\mu^{(r)}),(\lambda^{(1)}|\cdots|\lambda^{(r)}) \in \Par^r$ define $\cL_{(\mu^{(1)}|\cdots|\mu^{(r)})}(\lambda^{(1)}|\cdots|\lambda^{(r)})$ to be the set of labellings $\sigma \in \cL(\lambda^{(1)}|\cdots|\lambda^{(r)})$ such that for all $1\leq j \leq r$ and $1 \leq i \leq \ell(\mu^{(j)})$ 
    $$|\sigma^{-1}(i_j)| = \mu^{(j)}_i.$$
\end{defn}

We may alternatively describe $\cL_{(\mu^{(1)}|\cdots|\mu^{(r)})}(\lambda^{(1)}|\cdots|\lambda^{(r)})$ as the set of all allowed fillings with \textbf{\textit{shape}} $(\lambda^{(1)}|\cdots|\lambda^{(r)})$ and \textbf{\textit{content}} $(\mu^{(1)}|\cdots|\mu^{(r)}).$ We readily find the following combinatorial interpretation for the multi-symmetric Kostka numbers.

\begin{thm}\label{Kostka coeff thm}
    $$K_{(\mu^{(1)}|\cdots|\mu^{(r)})}^{(\lambda^{(1)}|\cdots|\lambda^{(r)})} = |\cL_{(\mu^{(1)}|\cdots|\mu^{(r)})}(\lambda^{(1)}|\cdots|\lambda^{(r)})|$$
\end{thm}
\begin{proof}
    This follows immediately from Theorem \ref{combinatorial formula for multi-sym Schur} by applying the symmetries in each variable set $X_1,\ldots, X_r.$ 
\end{proof}

\begin{example}
    We saw previously in Example \ref{2-multisym example of key poly} that
    $$\cS_{(2|1)} = s_{(2|1)} + s_{(2,1|\emptyset)}.$$ By expanding into the monomial basis we find 
    $$\cS_{(2|1)} = m_{(2|1)}+ m_{(1,1|1)}+ m_{(2,1|\emptyset)} + 2m_{(1,1,1|\emptyset)}.$$ These coefficients may be computed diagrammatically using Theorem \ref{Kostka coeff thm} as follows:
    \begin{multicols}{2}
    \begin{itemize}
        \item 
    $(2|1) \rightarrow $ ~~~~  \ytableausetup
 {mathmode, boxframe=normal, boxsize=2em}\begin{ytableau}
       *(red) 1_1  &  \none \\
        *(red) 1_1 & *(yellow) 1_2  \\
        *(red) \omega_1 & *(yellow) \omega_2 \\
        \end{ytableau}
        
        \item 
    $(1,1|1) \rightarrow $ ~~~~ \begin{ytableau}
       *(red) 1_1  &  \none \\
        *(red) 2_1 & *(yellow) 1_2  \\
         *(red) \omega_1 & *(yellow) \omega_2 \\
        \end{ytableau}
        
         \item 
    $(2,1|\emptyset) \rightarrow $ ~~~~ \begin{ytableau}
       *(red) 1_1  &  \none \\
        *(red) 1_1 & *(yellow) 2_1  \\
         *(red) \omega_1 & *(yellow) \omega_2 \\
        \end{ytableau}
        
         \item 
    $(1,1,1|\emptyset) \rightarrow $ ~~~~ \begin{ytableau}
       *(red) 1_1  &  \none \\
        *(red) 2_1 & *(yellow) 3_1 \\
         *(red) \omega_1 & *(yellow) \omega_2 \\
        \end{ytableau}~~,~~ \begin{ytableau}
       *(red) 2_1  &  \none \\
        *(red) 3_1 & *(yellow) 1_1  \\
         *(red) \omega_1 & *(yellow) \omega_2 \\
        \end{ytableau}
        \end{itemize}
        \end{multicols}
Notice that the diagram filling 
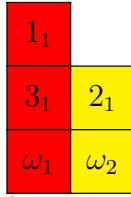
\begin{ytableau}
       *(red) 1_1  &  \none \\
        *(red) 3_1 & *(yellow) 2_1 \\
         *(red) \omega_1 & *(yellow) \omega_2 \\
        \end{ytableau}
is \textbf{not} in $\cL_{(1,1,1|\emptyset)}(2|1)$ since the entries $1_1,2_1,3_1$ increase clockwise and therefore form a Type 1 co-inversion triple. Furthermore, we may confirm that $K_{(1|2)}^{(2|1)} = 0$ by noticing that any theoretical labeling in $\cL_{(1|2)}(2|1)$ must have at least one of the labels $1_2,2_2$ in the first column which is impossible since $\omega_1< 1_2,2_2.$
\end{example}

Using Theorem \ref{Kostka coeff thm}, we may explicitly compute the \textit{leading} Kostka coefficient of each multi-symmetric Schur function. 

\begin{cor}\label{leading coefficient is 1}
    $$K_{(\lambda^{(1)}|\cdots|\lambda^{(r)})}^{(\lambda^{(1)}|\cdots|\lambda^{(r)})} = 1$$
\end{cor}
\begin{proof}
    From Theorem \ref{Kostka coeff thm}, we know that $K_{(\lambda^{(1)}|\cdots|\lambda^{(r)})}^{(\lambda^{(1)}|\cdots|\lambda^{(r)})} = |\cL_{(\lambda^{(1)}|\cdots|\lambda^{(r)})}(\lambda^{(1)}|\cdots|\lambda^{(r)})|$ so it suffices to show that $\cL_{(\lambda^{(1)}|\cdots|\lambda^{(r)})}(\lambda^{(1)}|\cdots|\lambda^{(r)})$ contains exactly one filling. First, the labels from the first alphabet $\{1_1,2_1,\ldots\}$ must all lay in the $\rev(\lambda^{(1)})$-component of the diagram. Otherwise, if any of the labels $\{1_1,2_1,\ldots\}$ lay in the other components of the diagram then some label $i_j$ for $j \geq 2$ must be in the $\rev(\lambda^{(1)})$-component which contradicts the weakly-decreasing requirement. There is exactly one way to fill the shape $\rev(\lambda^{(1)})$ with content $\lambda^{(1)}$ since the labeling must be inverse semi-standard and for the usual Kostka numbers we know that $K_{\nu,\nu} =1$ for all partitions $\nu$. Now that we have accounted for the labels $\{1_1,2_1,\ldots\}$ along with the sub-diagram $\rev(\lambda^{(1)})$, we see that the remaining labels follow the same pattern. In the end, for all $1\leq i \leq r$ the labels $\{1_i,2_i,\ldots\}$ must fill the $\rev(\lambda^{(i)})$-component in the unique way with content $\lambda^{(i)}.$ Since there is only one way to do this, we see that $\cL_{(\lambda^{(1)}|\cdots|\lambda^{(r)})}(\lambda^{(1)}|\cdots|\lambda^{(r)})$ contains only one filling.
\end{proof}

\begin{example}
    The unique element of $\cL_{(2,1|2,1,1|4,4,4,3,2,2,2)}(2,1|2,1,1|4,4,4,3,2,2,2)$ is the following:
\begin{center}
\ytableausetup
 {mathmode, boxframe=normal, boxsize=3em}
\begin{ytableau}
          \none & \none & \none & \none & \none & \none & \none & \none & \none & *(green) 1_3 &*(green) 2_3 &*(green) 3_3\\
         \none & \none & \none & \none & \none & \none & \none & \none & *(green) 1_3 & *(green)2_3 &*(green) 3_3 &*(green)4_3 \\
         \none & *(red) 1_1 & \none & \none & *(yellow) 1_2 & *(green) 1_3 & *(green) 2_3 & *(green) 3_3 & *(green) 4_3 & *(green) 5_3 & *(green)6_3 & *(green) 7_3\\
           *(red) 1_1 & *(red) 2_1 & *(yellow) 1_2 & *(yellow) 2_2 & *(yellow) 3_2 & *(green) 1_3 & *(green) 2_3 & *(green) 3_3 & *(green) 4_3 & *(green) 5_3 & *(green) 6_3 &  *(green) 7_3\\ 
           *(red) \omega_1 & *(red) \omega_1 + 1 & *(yellow) \omega_2 & *(yellow) \omega_2 + 1 & *(yellow) \omega_2 +2 & *(green) \omega_3 & *(green) \omega_3 +1 & *(green) \omega_3 + 2 & *(green) \omega_3 + 3 & *(green) \omega_3 + 4 & *(green) \omega_3 + 5 & *(green) \omega_3 +6  \\
\end{ytableau}
\end{center}
    
\end{example}

\subsection{Triangularity}

In this section, we utilize partial orderings on multi-partitions in order to give triangular expansions for the multi-symmetric Schur functions. We start by using a crude, but simple to use, lexicographic order to give a triangular expansion for the multi-symmetric Schur functions (Theorem \ref{crude triangularity thm}). Using, Cor. \ref{leading coefficient is 1} we conclude that the multi-symmetric Schur functions are a $\mathbb{Z}$-basis for the ring of multi-symmetric functions over $\mathbb{Z}.$ Afterwards, we use a more subtle result of Fan-Guo-Peng-Sun (Theorem \ref{monomials in key polys thm}) to give a complete description of the monomial multi-symmetric supports of the multi-symmetric Schur functions (Theorem \ref{complete support of multi-sym Schur}). 

\begin{defn}
    For $\alpha,\beta \in \mathbb{Z}^n_{\geq 0}$ define the reverse lexicographic ordering $<_{RL}$ by $\beta <_{RL} \alpha$ if there exists some $1\leq k \leq n$ such that $\beta_i = \alpha_i$ if $i > k$ and $\beta_k < \alpha_{k}.$ For all $n_1,\ldots,n_r \geq 1$ and $(\alpha^{(1)}|\ldots | \alpha^{(r)}),(\beta^{(1)}|\ldots | \beta^{(r)}) \in \mathbb{Z}_{\geq 0}^{n_1}\times \cdots \times \mathbb{Z}_{\geq 0}^{n_r}$ we write $(\beta^{(1)}|\ldots | \beta^{(r)}) \leq_{RL} (\alpha^{(1)}|\ldots | \alpha^{(r)})$ if $\beta^{(1)}*\cdots * \beta^{(r)} \leq_{RL} \alpha^{(1)}*\cdots * \alpha^{(r)} $ or equivalently there exists some $1 \leq k \leq r$ such that for all $i > k,$ $ \alpha^{(i)} = \beta^{(i)}$ and $\alpha^{(k)} \leq_{RL} \beta^{(k)}.$
\end{defn}

The following triangularity result is due to Reiner-Shimozono:

\begin{prop}\cite{RS95}
For all $\alpha \in \mathbb{Z}_{\geq 0}^n$, there exist non-negative integers $c_{\alpha,\beta}$ such that 
    $$\cK_{\alpha} = x^{\alpha} + \sum_{\beta <_{RL} \alpha} c_{\alpha,\beta} x^{\beta}.$$
\end{prop}

Applying the above triangularity result to the multi-symmetric Schur polynomials we see:

\begin{cor}
For all $(\lambda^{(1)}|\cdots | \lambda^{(r)}) \in \Par^r$ and $(n_1|\ldots|n_r)$ with $n_i \geq \ell(\lambda^{(i)})$ for all $1 \leq i \leq r$,

    $$\cS_{(\lambda^{(1)}|\cdots | \lambda^{(r)})}^{(n_1|\cdots | n_r)}(X_1^{(n_1)}|\cdots | X_r^{(n_r)}) = \sum_{ } c_{(\mu^{(1)}|\cdots | \mu^{(r)}),(\lambda^{(1)}|\cdots | \lambda^{(r)})}^{(n_1|\ldots |n_r)} m_{(\mu^{(1)}|\cdots | \mu^{(r)})}(X_1^{(n_1)}|\cdots | X_r^{(n_r)})$$
    where the sum ranges over all $(\mu^{(1)}|\cdots | \mu^{(r)}) \in \Par^r$ with $$(0^{n_1-\ell(\mu^{(1)})}*\rev(\mu^{(1)})|\cdots | 0^{n_1-\ell(\mu^{(r)})}*\rev(\mu^{(r)})) \leq_{RL} (0^{n_1-\ell(\lambda^{(1)})}*\rev(\lambda^{(1)})|\cdots | 0^{n_1-\ell(\lambda^{(r)})}*\rev(\lambda^{(r)}))$$ and $c_{(\mu^{(1)}|\cdots | \mu^{(r)}),(\lambda^{(1)}|\cdots | \lambda^{(r)})}^{(n_1|\ldots |n_r)} \geq 0$.
\end{cor}

\begin{remark}
From the stability in Lemma \ref{mutlisym stability lemma} we see that
    $$\lim_{n_1,\ldots,n_r \rightarrow \infty} c_{(\mu^{(1)}|\cdots | \mu^{(r)}),(\lambda^{(1)}|\cdots | \lambda^{(r)})}^{(n_1|\ldots |n_r)} = K_{(\mu^{(1)}|\cdots | \mu^{(r)})}^{(\lambda^{(1)}|\cdots | \lambda^{(r)})}.$$
\end{remark}

\begin{defn}
For $\mu,\lambda \in \Par$ we define the lexicographic ordering $<_L$ via $\mu <_L \lambda$ if 
there exists some $k \geq 1$ such that for all $i < k$, $\mu_i = \lambda_i$ and $\mu_k < \lambda_k$ where we adjoin $0's$ to the end of partitions when necessary. For $(\mu^{(1)}|\ldots | \mu^{(r)}),(\lambda^{(1)}|\ldots | \lambda^{(r)}) \in \Par^r$ we define the ordering $\prec$ via $(\mu^{(1)}|\ldots | \mu^{(r)}) \prec (\lambda^{(1)}|\ldots | \lambda^{(r)})$ if there exists some $1\leq k \leq r$ such that for all $i > k$, $\mu^{(i)} = \lambda^{(i)}$ and $\mu^{(k)} <_L \lambda^{(k)}.$
\end{defn}

The multi-symmetric Schur functions are triangular in the monomial multi-symmetric functions with respect to the order $\prec$.

\begin{thm}\label{crude triangularity thm}
For all $(\lambda^{(1)}|\cdots | \lambda^{(r)}) \in \Par^r$ we have the expansion
$$\cS_{(\lambda^{(1)}|\cdots | \lambda^{(r)})} = m_{(\lambda^{(1)}|\cdots | \lambda^{(r)})} +\sum_{(\mu^{(1)}|\cdots | \mu^{(r)}) \prec (\lambda^{(1)}|\cdots | \lambda^{(r)}) } K_{(\mu^{(1)}|\cdots | \mu^{(r)})}^{(\lambda^{(1)}|\cdots | \lambda^{(r)})} m_{(\mu^{(1)}|\cdots | \mu^{(r)})}.$$
\end{thm}

We immediately conclude the following:

\begin{cor}\label{Z basis Cor.}
    The multi-symmetric Schur functions are a $\mathbb{Z}$-basis of $\Lambda_{\mathbb{Z}}(X_1|\cdots | X_r).$
\end{cor}

\begin{remark}
    The ordering $\prec$ is rather crude and does not fully describe the monomial support of each multi-symmetric Schur function even after accounting for degree. This is easily seen in the $r= 1$ case as the ordering $\prec$ does not agree with the dominance ordering for partitions. For example, $(2,2,2) \prec (3,1,1,1)$ but $(2,2,2)$ and $(3,1,1,1)$ are incomparable with respect to the dominance ordering. Diagrammatically, this is explained by the non-existence of a semi-standard filling of $(3,1,1,1)$ with content $(2,2,2).$
\end{remark}

In order to give a full description of the monomial supports of the multi-symmetric Schur functions, we will use a complete characterization for the monomial supports of key polynomials. 

\begin{defn}\cite{MTY19}\label{key moves defn}
    For $\alpha \in \mathbb{Z}_{\geq 0}^n$ and $1 \leq i < j \leq n$ let $t_{i,j}(\alpha)$ be the vector obtained by interchanging $\alpha_i$ and $\alpha_j$, and let $m_{i,j}(\alpha) := \alpha + e_i-e_j$ where for $1\leq k \leq n$ we let $e_k$ be the standard basis vector. For $\beta \in \mathbb{Z}_{\geq 0}^{n}$ define $\beta <_{\kappa} \alpha $ if $\beta$ can be obtained from $\alpha$ by a sequence of moves $t_{i,j}$ for $\alpha_i < \alpha_j$ and $m_{i,j}$ for $\alpha_i < \alpha_j -1.$ 
\end{defn}

The next triangularity result was first conjectured by Monical--Tokcan--Yong \cite{MTY19} and proven by Fan--Guo--Peng--Sun \cite{FGPS}. 

\begin{thm}\cite{FGPS}\label{monomials in key polys thm}
    For $\alpha \in \mathbb{Z}_{\geq 0}^{n}$, there exist positive integers $d_{\alpha,\beta} > 0$ such that
    $$\cK_{\alpha} = x^{\alpha} + \sum_{\beta <_{\kappa} \alpha} d_{\alpha, \beta} x^{\beta}.$$ 
\end{thm}

For the next lemma we assume partitions have infinite strings of $0$'s adjoined after their final nonzero entry.   

\begin{lem}\label{Moves lemma}
    For $(\mu^{(1)}|\cdots |\mu^{(r)}),(\lambda^{(1)}|\cdots | \lambda^{(r)}) \in \Par^r$ and $(n_1|\ldots|n_r)$ with $n_i \geq |\lambda^{(1)}| + \ldots + |\lambda^{(r)}|$ for all $1\leq i \leq r$, the monomial $x^{(\mu^{(1)}*0^{n_1- \ell(\mu^{(1)})})*\cdots *(\mu^{(r)}*0^{n_r- \ell(\mu^{(1)})})}$ appears with a nonzero coefficient in the monomial expansion of $\cS_{(\lambda^{(1)}|\cdots | \lambda^{(r)})}^{(n_1|\ldots|n_r)}(x_1,\ldots,x_{n_1},\ldots,x_{n_1+\ldots +n_{r-1}+1},\ldots,x_{n_1+\ldots+n_{r}})$ if and only if $(\mu^{(1)}|\cdots |\mu^{(r)})$ can be obtained from $(\lambda^{(1)}|\cdots | \lambda^{(r)})$ by a sequence of moves on multi-partitions $(\nu^{(1)}|\cdots|\nu^{(r)}) \in \Par^r$ of the following types:
    \begin{enumerate}
        \item[\textbf{Move 1}:] \label{Moves} If $\nu^{(i_1)}_{j_1}< \nu^{(i_2)}_{j_2}$ for $i_1< i_2$, then swap the entries $\nu^{(i_1)}_{j_1},\nu^{(i_2)}_{j_2}$ and re-sort each partition.
        \item[\textbf{Move 2}:]  If $\nu^{(i)}_{j} > \nu^{(i)}_{k}+1$ for $j< k$, then replace $\nu^{(i)}$ by $\nu^{(i)}+e_{k}-e_j$ and re-sort the partition.
        \item[\textbf{Move 3}:]  If $\nu_{j_1}^{(i_1)} +1 < \nu^{(i_2)}_{j_2}$ for $i_1 < i_2$, then replace the entries $\nu_{j_1}^{(i_1)},\nu^{(i_2)}_{j_2}$ by $\nu_{j_1}^{(i_1)}+1,\nu^{(i_2)}_{j_2}-1$ respectively and re-sort each partition. 
    \end{enumerate}
\end{lem}
\begin{proof}
    First, by Theorem \ref{monomials in key polys thm} the monomial $x^{(\mu^{(1)}*0^{n_1- \ell(\mu^{(1)})})*\cdots (\mu^{(r)}*0^{n_r- \ell(\mu^{(1)})})}$ appears with a nonzero coefficient in the monomial expansion of $\cS_{(\lambda^{(1)}|\cdots | \lambda^{(r)})}^{(n_1|\ldots|n_r)}(x_1,\ldots,x_{n_1},\ldots,x_{n_1+\ldots +n_{r-1}+1},\ldots,x_{n_1+\ldots+n_{r}})$ if and only if $(\mu^{(1)}*0^{n_1- \ell(\mu^{(1)})})*\cdots * (\mu^{(r)}*0^{n_r- \ell(\mu^{(1)})})$ may be obtained from $(0^{n_1- \ell(\lambda^{(1)})}*\rev(\lambda^{(1)}))*\cdots (0^{n_r-\ell(\lambda^{(r)})}*\rev(\lambda^{(r)}))$ using moves $t_{i,j}$ or $m_{i,j}$ from Definition \ref{key moves defn} for $i < j$ when appropriate. This allows, for instance, for one to sort any sub-string of entries into weakly decreasing order.  Therefore, since we are only considering the vector $(\mu^{(1)}*0^{n_1- \ell(\mu^{(1)})})*\cdots * (\mu^{(r)}*0^{n_r- \ell(\mu^{(1)})})$ we need only to concern ourselves with the entries that appear in each window $\{n_1+\ldots + n_{i-1}+1,\ldots, n_1+\ldots + n_{i-1}+n_i \}$ after applying the moves $t_{i,j}$ and $m_{i,j}$ without worrying about their order. Thus our options for arriving at $(\mu^{(1)}*0^{n_1- \ell(\mu^{(1)})})*\cdots * (\mu^{(r)}*0^{n_r- \ell(\mu^{(r)})})$ are exactly to 
    \begin{enumerate}
        \item use a move $t_{i,j}$ between different windows, 
        \item use a move $m_{i,j}$ within the same window, 
        \item or use a move $m_{i,j}$ between different windows.
    \end{enumerate}
    After re-sorting each window, these moves align exactly with Moves 1,2, and 3 respectively. Importantly, we have taken $n_1,\ldots, n_r$ to be sufficiently large so that any possible $(\nu^{(1)}|\cdots | \nu^{(r)}) \in \Par^r$ reachable using Moves 1,2, and 3 will actually appear as some $(\nu^{(1)}*0^{n_1- \ell(\nu^{(1)})})*\cdots * (\nu^{(r)}*0^{n_r- \ell(\nu^{(1)})})$ which can be obtained using the above $t_{i,j}$ and $m_{i,j}$ moves.
\end{proof}

We now may formalize these moves into a partial order on multi-partitions.

\begin{defn}
    We define the partial order $\triangleleft$ on $\Par^r$ by $(\mu^{(1)}|\cdots | \mu^{(r)}) \triangleleft (\lambda^{(1)}|\cdots | \lambda^{(r)})$ if $(\mu^{(1)}|\cdots | \mu^{(r)})$ may be obtained from $(\lambda^{(1)}|\cdots | \lambda^{(r)})$ by repeatedly applying Move 1, Move 2, or Move 3 from \ref{Moves} above.
\end{defn}

\begin{remark}
    For $r = 1$, the ordering $\triangleleft$ agrees with the usual dominance ordering on $\Par$ justifying our notational choice. It is clear that $\triangleleft$ does in fact define a partial order. In particular, if $(\mu^{(1)}|\cdots | \mu^{(r)}) \triangleleft (\lambda^{(1)}|\cdots | \lambda^{(r)})$, then we cannot also have $ (\lambda^{(1)}|\cdots | \lambda^{(r)}) \triangleleft (\mu^{(1)}|\cdots | \mu^{(r)})$ since Moves 1,2, and 3 all strictly lower the multi-partition with respect to the ordering $\prec.$ In other words, if $(\mu^{(1)}|\cdots | \mu^{(r)}) \triangleleft (\lambda^{(1)}|\cdots | \lambda^{(r)})$, then $(\mu^{(1)}|\cdots | \mu^{(r)}) \prec (\lambda^{(1)}|\cdots | \lambda^{(r)})$ but not necessarily the other way around. This relationship between orderings is expected because for $r =1$ the dominance order refines the lexicographic order on partitions.
\end{remark}

\begin{example}
    Starting with $(3,1|2,1,1|4,1,1|3)$ we see the following:
    \begin{enumerate}
        \item (Move 1) $(3,2|1,1,1|4,1,1|3) \triangleleft (3,1|2,1,1|4,1,1|3)$ 
        \item (Move 1) $(4,3,1|2,1,1|1,1|3) \triangleleft (3,1|2,1,1|4,1,1|3)$ 
        \item (Move 2) $(3,1|2,1,1|3,2,1|3) \triangleleft (3,1|2,1,1|4,1,1|3)$ 
        \item (Move 2) $(3,1|2,1,1|4,1,1|2,1) \triangleleft (3,1|2,1,1|4,1,1|3)$ 
        \item (Move 3) $(3,1|3,1,1|3,1,1|2,1) \triangleleft (3,1|2,1,1|4,1,1|3)$ 
        \item (Move 3) $(3,1,1|2,1,1|4,1,1|2) \triangleleft (3,1|2,1,1|4,1,1|3)$ 
    \end{enumerate}
\end{example}

Using Lemma \ref{Moves lemma}, we give a refined version of the triangularity in Theorem \ref{crude triangularity thm}.

\begin{thm}\label{complete support of multi-sym Schur}
For all $(\lambda^{(r)}|\cdots | \lambda^{(r)}) \in \Par^r$ 
$$\cS_{(\lambda^{(r)}|\cdots | \lambda^{(r)})} = m_{(\lambda^{(r)}|\cdots | \lambda^{(r)})} + \sum_{(\mu^{(1)}|\cdots | \mu^{(r)})\triangleleft (\lambda^{(r)}|\cdots | \lambda^{(r)})} K_{(\mu^{(1)}|\cdots | \mu^{(r)})}^{(\lambda^{(r)}|\cdots | \lambda^{(r)})} m_{(\mu^{(1)}|\cdots | \mu^{(r)})} $$ and furthermore each above Kostka-coefficient is \textit{nonzero}.
\end{thm}
\begin{proof}
  This follows from Lemma \ref{Moves lemma} since in that lemma we assumed that $n_1,\ldots,n_r$ are sufficiently large in order for the multi-symmetric monomial expansions of the multi-symmetric Schur functions to stabilize.
\end{proof}

\begin{remark}
    Equivalently, using Theorem \ref{Kostka coeff thm} we know that $\cL_{(\mu^{(1)}|\cdots|\mu^{(r)})}(\lambda^{(1)}|\cdots|\lambda^{(r)}) \neq \emptyset$ if and only if $(\mu^{(1)}|\cdots|\mu^{(r)}) \trianglelefteq (\lambda^{(1)}|\cdots|\lambda^{(r)}).$
\end{remark}

\begin{example}
    Starting from $(2,1|2|3)$ we see that $(3,1|2|2) \triangleleft (2,1|2|3)$ and $(2,1,1|1|3) \triangleleft (2,1|2|3)$ which we may verify diagrammatically from the following two allowable diagram fillings of shape $(2,1|2|3)$ with contents $(3,1|2|2)$ and $(2,1,1|1|3)$ respectively:
\begin{center}

\begin{ytableau}
       \none & \none & \none & *(green) 1_1 \\
       \none  & *(red) 1_1 & *(yellow) 1_2 & *(green) 1_3 \\
         *(red) 1_1 & *(red) 2_1  & *(yellow) 1_2 & *(green) 1_3 \\
         *(red) \omega_1 & *(red) \omega_1 + 1 & *(yellow) \omega_2 & *(green) \omega_3 \\
        \end{ytableau} ~~ 
        \begin{ytableau}
       \none & \none & \none & *(green) 1_3 \\
       \none  & *(red) 1_1 & *(yellow) 2_1 & *(green) 1_3 \\
         *(red) 1_1 & *(red) 2_1  & *(yellow) 1_2 & *(green) 1_3 \\
         *(red) \omega_1 & *(red) \omega_1 + 1 & *(yellow) \omega_2 & *(green) \omega_3\\
        \end{ytableau}
        
\end{center}

\end{example}

\section{Schur Expansion}

In this section, we show that the multi-symmetric Schur functions $\cS_{(\lambda^{(1)}|\cdots | \lambda^{(r)})}$ have non-negative expansions into the ordinary Schur functions $s_{(\mu^{(1)}|\cdots |\mu^{(r)})}.$ The main result of this section, Theorem \ref{Schur positivity theorem}, will follow from a simple representation theoretic statement about Demazure modules in type $\GL$ (Lemma \ref{parabolic Demazure lemma}). Recall this paper's notation regarding Demazure modules from Definition \ref{Demazure module defn}.

\begin{defn}
    Given $(n_1|\ldots | n_r)$ with each $n_i \geq 1$ define $\mathrm{P}(n_1|\ldots|n_r) \leq \mathrm{GL}_{n_1+\ldots+n_r}$ as the unique parabolic subgroup containing the Borel subgroup $\mathrm{B}_n$ and the Levi subgroup $L(n_1|\ldots|n_r):= \mathrm{GL}_{n_1}\times \cdots \times \mathrm{GL}_{n_r} \leq \mathrm{GL}_{n_1+\ldots+n_r}.$
\end{defn}

We will primarily be interested in a special family of Demazure modules.

\begin{defn}
    For $(\lambda^{(1)}|\cdots|\lambda^{(r)}) \in \Par^r$ and $(n_1|\ldots|n_r)$ with $n_i \geq \ell(\lambda^{(i)})$ for all $1\leq i \leq r$ define the Demazure module 
    $$\mathcal{W}_{(\lambda^{(1)}|\cdots | \lambda^{(r)})}^{(n_1|\ldots|n_r)}:= \mathcal{V}^{\sort(\lambda^{(1)}*\cdots *\lambda^{(r)})*0^{n_1+\ldots+n_r - (\ell(\lambda^{(1)})+\ldots + \ell(\lambda^{(r)}))}}_{(0^{n_1-\ell(\lambda^{(1)})}*\rev(\lambda^{(1)}))       *\cdots*(0^{n_r-\ell(\lambda^{(r)})}*\rev(\lambda^{(r)}))}$$
\end{defn}

The Demazure modules $\mathcal{W}_{(\lambda^{(1)}|\cdots | \lambda^{(r)})}^{(n_1|\ldots|n_r)}$ are directly related to the multi-symmetric Schur polynomials by the following:

\begin{prop}
$$\Char(\mathcal{W}_{(\lambda^{(1)}|\cdots | \lambda^{(r)})}^{(n_1|\ldots|n_r)}) = \cS_{(\lambda^{(1)}|\cdots | \lambda^{(r)})}^{(n_1|\ldots|n_r)}$$
\end{prop}
\begin{proof}
    This follows from the Demazure character formula \ref{Demazure Character Formula} and Definition \ref{multi-symmetric Schur poly defn}.
\end{proof}

We require the next simple lemma.

\begin{lem}\label{parabolic Demazure lemma}
    Let $\mu \in \mathbb{Z}_{\geq 0}^{n}$ such that the sub-string $\mu_{i+1},\ldots, \mu_{i+k}$ is weakly increasing. Then the Demazure module $\mathcal{V}^{\sort(\mu)}_{\mu}$ is a $\mathrm{GL}_1^{i}\times \mathrm{GL}_{k} \times \mathrm{GL}_{1}^{n-k-i}$-submodule of $\mathrm{Res}_{\mathrm{GL}_1^{i}\times \mathrm{GL}_{k} \times \mathrm{GL}_{1}^{n-k-i}}^{\mathrm{GL}_n} \mathcal{V}^{\sort(\mu)}.$ 
\end{lem}
\begin{proof}
    The argument is identical to that of Lemma 5.8 in \cite{BWAlmostSymSchur}.
\end{proof}

Applying Lemma \ref{parabolic Demazure lemma} to the Demazure modules $\mathcal{W}_{(\lambda^{(1)}|\cdots | \lambda^{(r)})}^{(n_1|\ldots|n_r)}$, we see that:

\begin{prop}\label{parabolic Demazure prop}
    $\mathcal{W}_{(\lambda^{(1)}|\cdots | \lambda^{(r)})}^{(n_1|\ldots|n_r)}$ is a $\mathrm{P}(n_1|\ldots|n_r)$- submodule of $\mathcal{V}^{\sort(\lambda^{(1)}*\cdots *\lambda^{(r)})*0^{n_1+\ldots+n_r - (\ell(\lambda^{(1)})+\ldots + \ell(\lambda^{(r)}))}}.$
\end{prop}
\begin{proof}
    The group $\mathrm{P}(n_1|\ldots|n_r)$ is the semi-direct product of $\mathrm{B}_{n_1+\ldots+n_r}$ and $L(n_1|\cdots |n_r)$ and we already know $\mathcal{W}_{(\lambda^{(1)}|\cdots | \lambda^{(r)})}^{(n_1|\ldots|n_r)}$ is a $\mathrm{B}_{n_1+\ldots+n_r}$-submodule of $\mathcal{V}^{\sort(\lambda^{(1)}*\cdots *\lambda^{(r)})*0^{n_1+\ldots+n_r - (\ell(\lambda^{(1)})+\ldots + \ell(\lambda^{(r)}))}}$. From Lemma \ref{parabolic Demazure lemma}, find that $\mathcal{W}_{(\lambda^{(1)}|\cdots | \lambda^{(r)})}^{(n_1|\ldots|n_r)}$ is also a $L(n_1|\cdots |n_r)$-submodule. The result follows.
\end{proof}

Here we define the irreducible polynomial representations of $L(n_1|\cdots | n_r)$.

\begin{defn}
    For all $(\mu^{(1)}|\cdots |\mu^{(r)}) \in \Par^r $ and $(n_1|\cdots|n_r)$ with $n_i \geq \ell(\mu^{(i)})$ for all $1 \leq i \leq r$ we let $\mathcal{D}^{(n_1|\ldots|n_r)}_{(\mu_1|\cdots|\mu_r)}$ denote the irreducible $L(n_1|\cdots | n_r)$-module 
    $$\mathcal{D}^{(n_1|\ldots|n_r)}_{(\mu_1|\cdots|\mu_r)}:= \mathcal{V}^{\mu^{(1)}*0^{n_1-\ell(\mu^{(1)})}}\otimes \cdots \otimes \mathcal{V}^{\mu^{(r)}*0^{n_r - \ell(\mu^{(r)})}}.$$
\end{defn}

Importantly, note that 
$$\Char \left( \mathcal{D}^{(n_1|\ldots|n_r)}_{(\mu_1|\cdots|\mu_r)} \right) (x_{1,1},\ldots, x_{1,n_1},\ldots, x_{r,1},\ldots,x_{r,n_r}) = s_{(\mu_1|\cdots|\mu_r)}(X_1^{(n_1)}|\cdots |X_{r}^{(n_r)}).$$

\begin{defn}
    Define the (a priori) rational numbers $(M_{(\mu^{(1)}|\cdots | \mu^{(r)})}^{(\lambda^{(1)}|\cdots | \lambda^{(r)})})_{(\mu^{(1)}|\cdots | \mu^{(r)}),(\lambda^{(1)}|\cdots | \lambda^{(r)}) \in \Par^r}$ by 
    $$\cS_{(\lambda^{(1)}|\cdots | \lambda^{(r)})} = \sum_{(\mu^{(1)}|\cdots | \mu^{(r)}) \in \Par^r} M_{(\mu^{(1)}|\cdots | \mu^{(r)})}^{(\lambda^{(1)}|\cdots | \lambda^{(r)})} s_{(\mu^{(1)}|\cdots | \mu^{(r)})}.$$
\end{defn}

\begin{thm} \label{Schur positivity theorem} 
 For all $(\mu^{(1)}|\cdots | \mu^{(r)}),(\lambda^{(1)}|\cdots | \lambda^{(r)}) \in \Par^r,$ $$M_{(\mu^{(1)}|\cdots | \mu^{(r)})}^{(\lambda^{(1)}|\cdots | \lambda^{(r)})} \in \mathbb{Z}_{\geq 0}.$$ In fact, for $(n_1|\cdots |n_r)$ with $n_i \geq |\lambda^{(1)}|+\ldots + |\lambda^{(r)}|$ for every $1\leq i \leq r,$ $$ M_{(\mu^{(1)}|\cdots | \mu^{(r)})}^{(\lambda^{(1)}|\cdots | \lambda^{(r)})} =  \dim \Hom_{L(n_1|\ldots|n_r)}(\mathcal{D}^{(n_1|\ldots|n_r)}_{(\mu_1|\cdots|\mu_r)}, \mathcal{W}_{(\lambda^{(1)}|\cdots| \lambda^{(r)})}^{(n_1|\ldots|n_r)} ).$$ 
\end{thm}
\begin{proof}
We start by determining the ordinary Schur coefficients of the multi-symmetric Schur polynomials. Let $(\lambda^{(1)}|\cdots | \lambda^{(r)}) \in \Par^r$ and  $(n_1|\cdots |n_r)$ with $n_i \geq |\lambda^{(1)}|+\ldots + |\lambda^{(r)}|$ for every $1\leq i \leq r.$ Using Proposition \ref{parabolic Demazure prop} and the fact that $L(n_1|\cdots|n_r)$ is reductive, we may decompose $\mathcal{W}_{(\lambda^{(1)}|\cdots | \lambda^{(r)})}^{(n_1|\ldots|n_r)}$ into irreducible $L(n_1|\cdots|n_r)$-submodules as 
$$\mathcal{W}_{(\lambda^{(1)}|\cdots | \lambda^{(r)})}^{(n_1|\ldots|n_r)} = \bigoplus_{(\mu^{(1)}|\cdots | \mu^{(r)})} \left( \mathcal{D}^{(n_1|\ldots|n_r)}_{(\mu_1|\cdots|\mu_r)}  \right) ^{\oplus d_{(\mu^{(1)}|\cdots | \mu^{(r)})}}$$ where the $d_{(\mu^{(1)}|\cdots | \mu^{(r)})}$ are some non-negative integer multiplicities.
By applying the character map we see
    \begin{align*}
        \cS_{(\lambda^{(1)}|\cdots | \lambda^{(r)})}^{(n_1|\ldots|n_r)} &= \Char(\mathcal{W}_{(\lambda^{(1)}|\cdots | \lambda^{(r)})}^{(n_1|\ldots|n_r)}) \\
        &= \Char \left( \bigoplus_{(\mu^{(1)}|\cdots | \mu^{(r)})} \left( \mathcal{D}^{(n_1|\ldots|n_r)}_{(\mu_1|\cdots|\mu_r)}  \right) ^{\oplus d_{(\mu^{(1)}|\cdots | \mu^{(r)})}}\right)\\ 
        &= \sum_{(\mu^{(1)}|\cdots | \mu^{(r)})} d_{(\mu^{(1)}|\cdots | \mu^{(r)})} \Char\left( \mathcal{D}^{(n_1|\ldots|n_r)}_{(\mu_1|\cdots|\mu_r)}  \right) \\
        &= \sum_{(\mu^{(1)}|\cdots | \mu^{(r)})} d_{(\mu^{(1)}|\cdots | \mu^{(r)})} s_{(\mu^{(1)}|\cdots|\mu^{(r)})}(X_1^{(n_1)}|\cdots |X_{r}^{(n_r)}).\\
    \end{align*}

   Since each $n_i$ is sufficiently large, we know that 
$$ M_{(\mu^{(1)}|\cdots | \mu^{(r)})}^{(\lambda^{(1)}|\cdots | \lambda^{(r)})} = d_{(\mu^{(1)}|\cdots | \mu^{(r)})} =   \dim \Hom_{L(n_1|\ldots|n_r)}(\mathcal{D}^{(n_1|\ldots|n_r)}_{(\mu_1|\cdots|\mu_r)}, \mathcal{W}_{(\lambda^{(1)}|\cdots| \lambda^{(r)})}^{(n_1|\ldots|n_r)} ).$$ The result follows. 
\end{proof}

\begin{remark}\label{LR remark}
    Theorem \ref{Schur positivity theorem} leads to two questions. First, is there a simple combinatorial rule to determine when $M_{(\mu^{(1)}|\cdots | \mu^{(r)})}^{(\lambda^{(1)}|\cdots | \lambda^{(r)})} \neq 0?$ Second, is there an explicit (positive) combinatorial formula for the $M_{(\mu^{(1)}|\cdots | \mu^{(r)})}^{(\lambda^{(1)}|\cdots | \lambda^{(r)})}$ coefficients? The second question has an answer due to Ross--Yong. Using Theorem 1.1 of \cite{RY15}, which involves the \textbf{Edelman--Greene correspondence}, we see that the coefficients $M_{(\mu^{(1)}|\cdots | \mu^{(r)})}^{(\lambda^{(1)}|\cdots | \lambda^{(r)})}$ are counted by the number of sequences of strictly increasing tableaux $(T_1,\dots,T_r)$ with shapes $\mu^{(1)},\dots, \mu^{(r)}$ respectively satisfying a few additional constraints which we will not describe in detail here. Explicitly in terms of the notation of Ross--Yong, $M_{(\mu^{(1)}|\cdots | \mu^{(r)})}^{(\lambda^{(1)}|\cdots | \lambda^{(r)})} = \mathcal{E}^{\alpha}_{\mu^{(1)},\dots, \mu^{(r)}}$ where 
    $\alpha:= (0^{N-\ell(\lambda^{(1)})}*\rev(\lambda^{(1)}))*\cdots * (0^{N-\ell(\lambda^{(r)})}*\rev(\lambda^{(r)}))$ and $N:= |\lambda^{(1)}|+\dots + |\lambda^{(r)}|.$ 
\end{remark}

\section{Plethystic Recurrence Relation}

In this final section, we will prove an interesting plethystic identity (Corollary \ref{plethysm formula for multi-sym Schurs}) for the multi-symmetric Schur functions relating to the Bernstein operators. 

\subsection{Set-up}

\begin{defn}
For the remainder of the paper, the variable $z$ will denote a free variable independent of any other variable sets. We will write $\langle z^a \rangle \sum_{n \in \mathbb{Z}} c_n z^n:= c_a$ for any formal series in $z.$ 
\end{defn}

\begin{defn}
 To simplify notation, for $1\leq i \leq r$ and $1\leq j \leq n_{i}-1$ we will write $$\xi^{(i)}_{j}: \mathbb{Q}[x_{1,1},\ldots,x_{1,n_1},\ldots,x_{r,1},\ldots,x_{r,n_r}] \rightarrow \mathbb{Q}[x_{1,1},\ldots,x_{1,n_1},\ldots,x_{r,1},\ldots,x_{r,n_r}]$$ for the divided difference operator $\xi^{(i)}_j:= \xi_{x_{i,j},x_{i,j+1}}.$ 
\end{defn}

The next lemma describes how to use isobaric divided difference operators to move parts of partitions in multi-symmetric Schur polynomials. 

\begin{lem}
\begin{align*}
    &\langle z^0 \rangle\xi^{(i+1)}_{n_{i+1}-1}\cdots \xi^{(i+1)}_{1} \xi_{z,x_{i+1,1}}\left( \cS_{(\lambda^{(1)}|\cdots |\lambda^{(r)})}^{(n_1|\cdots|n_r)} (X_1^{(n_1)}|\cdots| X_{i}^{(n_i)}+z|\cdots |X_r^{(n_r)}) \right)\\
    &= \cS_{(\lambda^{(1)}|\cdots|\lambda^{(i)}\setminus \{\lambda^{(i)}_{1}\}|\lambda^{(i+1)} \cup \{\lambda^{(i)}_{1}\}|\cdots |\lambda^{(r)})}^{(n_1|\cdots|n_r)} (X_1^{(n_1)}|\cdots|X_r^{(n_r)})\\
\end{align*}
   
\end{lem}
\begin{proof}
    Directly from the key polynomial recurrence relations we find that
    \begin{align*}
        &\xi^{(i+1)}_{n_{i+1}-1}\cdots \xi^{(i+1)}_{1} \xi_{z,x_{i+1,1}}\left( \cS_{(\lambda^{(1)}|\cdots |\lambda^{(r)})}^{(n_1|\cdots|n_r)} (X_1^{(n_1)}|\cdots| X_{i}^{(n_i)}+z|\cdots |X_r^{(n_r)}) \right)= \\
        & \xi^{(i+1)}_{n_{i+1}-1}\cdots \xi^{(i+1)}_{1} \xi_{z,x_{i+1,1}} \\
        & \left( \cK_{\cdots * (0^{n_{i}+1-\ell(\lambda^{(i)})}*\rev(\lambda^{(i)}))*(0^{n_{i+1}-\ell(\lambda^{(i+1)})}*\rev(\lambda^{(i+1)}))*\cdots}(\ldots,x_{i,1},\ldots,x_{i,n_{i}},z,x_{i+1,1},\ldots, x_{i+1,n_{i+1}},\ldots) \right) \\
        &= \cK_{\cdots * (0^{n_{i}-\ell(\lambda^{(i)}\setminus \{\lambda^{(i)}_{1}\})}*\rev(\lambda^{(i)}\setminus \{\lambda^{(i)}_{1}\})*(0^{n_{i+1}+1-\ell(\lambda^{(i+1)} \cup \{\lambda^{(i)}_{1}\})}*\rev(\lambda^{(i+1)} \cup \{\lambda^{(i)}_{1}\}))*\cdots}(\ldots,z,x_{i+1,1},\ldots, x_{i+1,n_{i+1}},\ldots) \\
        &= \cS_{(\lambda^{(1)}|\cdots|\lambda^{(i)}\setminus \{\lambda^{(i)}_{1}\}|\lambda^{(i+1)} \cup \{\lambda^{(i)}_{1}\}|\cdots |\lambda^{(r)})}^{(n_1|\cdots|n_{i+1}+1|\cdots |n_r)} (X_1^{(n_1)}|\cdots|X_{i+1}^{(n_{i+1})}+z|\cdots |X_r^{(n_r)}).\\
    \end{align*}

Now by evaluating at $z = 0$ (or equivalently extracting the $z^0$ coefficient) we see from Lemma \ref{mutlisym stability lemma} that

\begin{align*}
    &\langle z^0 \rangle \cS_{(\lambda^{(1)}|\cdots|\lambda^{(i)}\setminus \{\lambda^{(i)}_{1}\}|\lambda^{(i+1)} \cup \{\lambda^{(i)}_{1}\}|\cdots |\lambda^{(r)})}^{(n_1|\cdots|n_{i+1}+1|\cdots |n_r)} (X_1^{(n_1)}|\cdots|X_{i+1}^{(n_{i+1})}+z|\cdots |X_r^{(n_r)}) \\
    &= \cS_{(\lambda^{(1)}|\cdots|\lambda^{(i)}\setminus \{\lambda^{(i)}_{1}\}|\lambda^{(i+1)} \cup \{\lambda^{(i)}_{1}\}|\cdots |\lambda^{(r)})}^{(n_1|\cdots|n_{i+1}|\cdots |n_r)} (X_1^{(n_1)}|\cdots|X_{i+1}^{(n_{i+1})}|\cdots |X_r^{(n_r)}).\\
\end{align*}

\end{proof}

\begin{defn}\label{defn of finite var Ds}
    Define the operators $D_1^{(n_1|\cdots|n_r)},\ldots,D_{r-1}^{(n_{1}|\cdots | n_r)}: \Lambda(X_{1}^{(n_1)}|\cdots|X_{r}^{(n_r)}) \rightarrow \Lambda(X_{1}^{(n_1)}|\cdots|X_{r}^{(n_r)})$ by 
    $$D_{i}^{(n_1|\cdots|n_r)}(f):= \langle z^0 \rangle\xi^{(i+1)}_{n_{i+1}-1}\cdots \xi^{(i+1)}_{1} \xi_{z,x_{i+1,1}}\left( f (X_1^{(n_1)}|\cdots| X_{i}^{(n_i)}+z|\cdots |X_r^{(n_r)}) \right)$$
\end{defn}

The above operators satisfy the following stability property.

\begin{prop}
    For all $(n_1|\cdots|n_r)$, $1\leq i \leq r-1$ and $1 \leq j \leq r,$
    $$\Xi_{j}^{(n_j)}D_{i}^{(n_1|\cdots|n_r)}  = D_i^{(n_1|\cdots|n_j-1|\cdots|n_r)}\Xi_{j}^{(n_j)}.$$ 
\end{prop}
\begin{proof}
    This is a consequence of Lemma \ref{Ds into Ws lemma} proven independently later.
\end{proof}

Using the above stability result we may define the following operators on multi-symmetric functions:

\begin{defn}
    Define the operators $D_1,\ldots,D_{r-1}: \Lambda(X_1|\cdots|X_r)\rightarrow \Lambda(X_1|\cdots|X_r) $ by 
    $$D_i := \lim_{n_1,\ldots,n_r \rightarrow \infty} D_{i}^{(n_1|\cdots|n_r)}.$$
\end{defn}

Putting everything together yields:

\begin{prop}\label{recurrence prop}
    $$D_i\left(\cS_{(\lambda^{(1)}|\cdots |\lambda^{(r)})} \right) = \cS_{(\lambda^{(1)}|\cdots|\lambda^{(i)}\setminus \{\lambda^{(i)}_{1}\}|\lambda^{(i+1)} \cup \{\lambda^{(i)}_{1}\}|\cdots |\lambda^{(r)})}$$
\end{prop}

\begin{remark}
    It is important to note it does \textit{not} follow from Proposition \ref{recurrence prop} that the operators $D_1,\ldots,D_{r-1}$ generate every $\cS_{(\lambda^{(1)}|\cdots | \lambda^{(r)})}$ from the ordinary Schur functions $\cS_{(\lambda|\emptyset|\cdots |\emptyset)} = s_{\lambda}(X_1)$. There is no way to generate $(2|3|1)$ from any $(\mu|\emptyset| \emptyset)$ simply by moving largest parts rightward. 
\end{remark}

\subsection{Weyl Symmetrizers and Bernstein Operators}

For the remainder of the paper, we will focus on finding an explicit plethystic expression for the action of the operators $D_i.$ First, we need to relate the operators $D_1,\ldots,D_{r-1}$ to the Weyl symmetrization operators.

\begin{defn}
    Given a commutative ring $R$ and an ordered set of free variables $(y_1,\ldots,y_n)$ define the \textbf{\textit{Weyl symmetrizer}} in the variables $(y_1,\ldots,y_n)$ to be the map $W_{y_1,\ldots,y_n}:R[y_1,\ldots,y_n] \rightarrow R[y_1,\ldots,y_n]^{\mathfrak{S}_n}$ given by 
    $$W_{y_1,\ldots,y_n}(f(y_1,\ldots,y_n)):= \sum_{\sigma \in \mathfrak{S}_n} \sigma \left( f(y_1,\ldots, y_n) \prod_{1 \leq i < j \leq n} \left( \frac{1}{1-\frac{y_j}{y_i}}\right) \right).$$
\end{defn}

The map $W_{y_1,\ldots,y_n}$ is in fact defined over $R$ since we have the alternative formula:

$$W_{y_1,\ldots,y_n}(f(y_1,\ldots,y_n)) = \frac{ \sum_{\sigma \in \mathfrak{S}_{n}} (-1)^{\ell(\sigma)} \sigma \left( y_1^{n-1}\cdots y_{n-1}^1 f(y_1,\ldots, y_n)  \right)}{\prod_{1\leq i< j \leq n} \left( y_i-y_j \right)}.$$

The following is a well known relationship between the Weyl symmetrizers and the isobaric divided difference operators.

\begin{lem}\label{Weyl sym is isobaric for longest word}
    $$W_{y_1,\ldots,y_n} = (\xi_{y_{n-1},y_n}\cdots \xi_{y_1,y_2})(\xi_{y_{n-1},y_n}\cdots \xi_{y_2,y_3})\cdots (\xi_{y_{n-1},y_n}\xi_{y_{n-2},y_{n-1}})\xi_{y_{n-1},y_n}.$$
\end{lem}

The Weyl symmetrizers satisfy a useful stability property.

\begin{lem}
    For all $f \in R[y_1,\ldots,y_n],$
    $$W_{y_1,\ldots,y_n}(f(y_1,\ldots,y_{n-1},y_n))|_{y_n=0} = W_{y_1,\ldots,y_{n-1}}(f(y_1,\ldots,y_{n-1},0)).$$
\end{lem}
\begin{proof}
    This is a standard result. For a direct computational proof see Lemma 2.8.3 in \cite{BWThesis}.
\end{proof}

From the classical theory of symmetric functions we have the following important family of maps which serve as creation operators for the Schur functions.

\begin{defn}
    Let $R$ be a commutative ring and let $Y=y_1+y_2+\ldots$ be an infinite set of free variables. For $a \in\mathbb{Z}$, the \textbf{\textit{Bernstein operator}} is the map $\cB_a:\Lambda_R(Y) \rightarrow \Lambda_R(Y)$ given by 
    $$\cB_a(f)(Y):= \langle z^a \rangle f(Y-z^{-1})\Omega(zY).$$ 
\end{defn}

The Bernstein operators are directly related to the Weyl symmetrizers.

\begin{prop}\label{Weyls limit to Bernstein prop}
For all $f(Y) \in \Lambda_{R}(Y),$
    $$\lim_{n \rightarrow \infty} \langle z^0 \rangle W_{z,y_1,y_2,\ldots,y_n}(z^af(y_1,\ldots,y_n)) = \cB_a(f)(Y) = \langle z^0 \rangle z^a f(Y-z^{-1}) \Omega(zY).$$
\end{prop}
\begin{proof}
    This result is standard but in particular follows from setting $t=0$ in the last part of Lemma 2.14 in \cite{BWAlmostSymSchur}.
\end{proof}

In the finite variable setting, we may describe the operators $D_i^{(n_1|\cdots|n_r)}$ directly in terms of a plethysm and a Weyl operator.

\begin{lem}\label{Ds into Ws lemma}
For all $1\leq i \leq r-1,$ $n_1,\ldots , n_r \geq 1$, and $f\in \Lambda(X_1^{(n_1)}|\cdots | X_r^{(n_r)}),$
    $$D_i^{(n_1|\cdots|n_r)}(f(X_1^{(n_1)}|\cdots | X_r^{(n_r)})) = \langle z^0\rangle W_{z,x_{i+1,1},\ldots,x_{i+1,n_{i+1}}}(f(X_1^{(n_1)}|\cdots| X_i^{(n_i)}+z | \cdots | X_r^{(n_r)}))$$
\end{lem}
\begin{proof}
    This is a consequence of applying Lemma \ref{Weyl sym is isobaric for longest word} to Definition \ref{defn of finite var Ds}.
\end{proof}

By limiting the number of variables in Lemma \ref{Ds into Ws lemma}, we may show the following pleasant plethystic result. 

\begin{thm}\label{plethysm thm}
For all $1\leq i \leq r-1$ and $f \in \Lambda(X_1|\cdots|X_r),$
    $$D_i(f(X_1|\cdots|X_r))= \langle z^0 \rangle f(X_1|\cdots | X_i+z|X_{i+1}-z^{-1}|\cdots |X_r) \Omega(zX_{i+1})$$
\end{thm}
\begin{proof}
    We may assume without loss of generality that $f \in \Lambda(X_i|X_{i+1})$ since $D_i$ commutes with multiplication by symmetric functions in the variables sets $X_j$ for $j \neq i,i+1.$ We know that 
    $$D_i(f(X_i|X_{i+1})) = \lim_{n_1,\ldots,n_r \rightarrow \infty} D_i^{(n_1|\cdots|n_r)}(f(X_{i}^{(n_i)}|X_{i+1}^{(n_{i+1})})).$$ We now use Lemma \ref{Ds into Ws lemma} with $R = \Lambda(X_i)$ to see 
    $$D_i^{(n_1|\cdots|n_r)}(f(X_{i}^{(n_i)}|X_{i+1}^{(n_{i+1})})) = \langle z^0\rangle W_{z,x_{i+1,1},\ldots,x_{i+1,n_{i+1}}}(f( X_i^{(n_i)}+z | X_{i+1}^{(n_{i+1})})).$$ Let us expand $f(X+z|Y) = \sum_{\ell \geq 0} z^{\ell} f_{\ell}(X)g_{\ell}(Y)$ so that 
    $$\langle z^0\rangle W_{z,x_{i+1,1},\ldots,x_{i+1,n_{i+1}}}(f( X_i^{(n_i)}+z | X_{i+1}^{(n_{i+1})})) = \sum_{\ell \geq 0} f_{\ell}(X_i^{(n_i)}) \langle z^0\rangle \left(W_{z,x_{i+1,1},\ldots,x_{i+1,n_{i+1}}} z^{\ell}g_{\ell}(X_{i+1}^{(n_{i+1})}) \right).$$ Thus 
    $$\lim_{n_1,\ldots,n_r \rightarrow \infty} D_i^{(n_1|\cdots|n_r)}(f(X_{i}^{(n_i)}|X_{i+1}^{(n_{i+1})})) = \sum_{\ell \geq 0} f_{\ell}(X_i) \lim_{n_1,\ldots,n_r \rightarrow \infty} \langle z^0\rangle \left(W_{z,x_{i+1,1},\ldots,x_{i+1,n_{i+1}}} z^{\ell}g_{\ell}(X_{i+1}^{(n_{i+1})}) \right)$$ and so from Proposition \ref{Weyls limit to Bernstein prop} we find 
    $$\lim_{n_1,\ldots,n_r \rightarrow \infty} \langle z^0\rangle \left(W_{z,x_{i+1,1},\ldots,x_{i+1,n_{i+1}}} z^{\ell}g_{\ell}(X_{i+1}^{(n_{i+1})}) \right) = \mathcal{B}_{\ell}(g_{\ell})(X_{i+1})= \langle z^0 \rangle z^{\ell} g_{\ell}(X_{i+1}-z^{-1})\Omega(zX_{i+1}).$$ All together, 
    \begin{align*}
        D_i(f(X_i|X_{i+1})) &= \sum_{\ell \geq 0} f_{\ell}(X_i)\langle z^0 \rangle z^{\ell} g_{\ell}(X_{i+1}-z^{-1})\Omega(zX_{i+1})\\
        &= \langle z^0 \rangle \left( \sum_{\ell \geq 0} z^{\ell}f_{\ell}(X_i) g_{\ell}(X_{i+1}-z^{-1}) \right)\Omega(zX_{i+1})\\
        &= \langle z^0 \rangle f(X_{i}+z|X_{i+1}-z^{-1}) \Omega(zX_{i+1}).\\
    \end{align*}

\end{proof}

Lastly, we apply Theorem \ref{plethysm thm} to Proposition \ref{recurrence prop} to find:

\begin{cor}\label{plethysm formula for multi-sym Schurs}
For all $(\lambda^{(1)}|\cdots |\lambda^{(r)}) \in \Par^r$ and $1\leq i \leq r$ with $\lambda^{(i)} \neq \emptyset,$
    $$\cS_{(\lambda^{(1)}|\cdots|\lambda^{(i)}\setminus \{\lambda^{(i)}_{1}\}|\lambda^{(i+1)} \cup \{\lambda^{(i)}_{1}\}|\cdots |\lambda^{(r)})} = \langle z^0 \rangle \cS_{(\lambda^{(1)}|\cdots |\lambda^{(r)})}(X_1|\cdots | X_i+z|X_{i+1}-z^{-1}|\cdots |X_r) \Omega(zX_{i+1}) $$
\end{cor}

\begin{remark}
    The natural $q$-analogues of the operators $D_1,\ldots, D_{r-1}:\Lambda_{\mathbb{Q}}(X_1|\cdots|X_r) \rightarrow \Lambda_{\mathbb{Q}}(X_1|\cdots|X_r)$ are the operators $\mathcal{D}_1,\ldots, \mathcal{D}_{r-1}:\Lambda_{\mathbb{Q}(q)}(X_1|\cdots |X_r) \rightarrow \Lambda_{\mathbb{Q}(q)}(X_1|\cdots |X_r)$ given by 
    $$\mathcal{D}_i(f(X_1|\cdots|X_r))= \langle z^0 \rangle f(X_1|\cdots | X_i+z|X_{i+1}-z^{-1}|\cdots |X_r) \Omega((1-q)zX_{i+1}).$$ Many of the above results follow similarly in the $q$-analogous situation where we replace the isobaric divided difference operators with Demazure-Lusztig operators, the Weyl symmetrizers with the (normalized) Hecke symmetrizers, and the Bernstein operators with the Jing operators. We may also define vertex operators $\mathcal{D}_1(z),\ldots,\mathcal{D}_{r-1}(z) :\Lambda_{\mathbb{Q}(q)}(X_1|\cdots |X_r) \rightarrow \Lambda_{\mathbb{Q}(q)}(X_1|\cdots |X_r)((z))$ by 
    $$\mathcal{D}_i(z)f(X_1|\cdots |X_r) := f(X_1|\cdots | X_i+z|X_{i+1}-z^{-1}|\cdots |X_r) \Omega((1-q)zX_{i+1}).$$ Note that $\mathcal{D}_i = \langle z^0 \rangle \mathcal{D}_i(z)$ and $D_i = \langle z^0 \rangle \mathcal{D}_i(z)|_{q=0}.$ It is an interesting, and seemingly nontrivial, question to determine the algebraic relations the vertex operators $\mathcal{D}_1(z),\ldots,\mathcal{D}_{r-1}(z)$ satisfy. It would also be interesting to find more plethystic operators similar to the $D_1,\ldots,D_{r-1}$ which operate nicely on the multi-symmetric Schur functions.
\end{remark}

\printbibliography

\end{document}